\documentclass[parskip=full, titlepage=false, twoside=false]{scrartcl} 
\usepackage[utf8]{inputenc} 
\usepackage[T1]{fontenc} 
\usepackage[english]{babel}
\usepackage[autostyle=true]{csquotes}
\usepackage{lmodern}
\usepackage[pdfauthor={Jonathan Glöckle}, pdftitle={Initial data rigidity implies spacetime rigidity}]{hyperref} 
\usepackage{amsmath, amsthm, amssymb, bm, bbm, mathrsfs, mathtools}
\usepackage{extarrows}
\usepackage{tabulary,tabularx}
\usepackage{graphicx}
\usepackage{tikz-cd}
\usepackage{enumerate}
\usepackage[capitalise]{cleveref}
\usepackage{bbold}

\theoremstyle{plain} 
\newtheorem{Satz}{Theorem}[section] 
\newtheorem{SATZ}{Main Theorem} 
\newtheorem{Prop}[Satz]{Proposition} 
\newtheorem{Lem}[Satz]{Lemma}
\newtheorem{Kor}[Satz]{Corollary} 
\theoremstyle{definition}

\newtheorem{Bem}[Satz]{Remark}

\newtheorem{Set}[Satz]{Setup}
\newtheorem{Frag}[Satz]{Question}

\crefname{Satz}{Theorem}{Theorems}
\crefname{Prop}{Proposition}{Propositions}
\crefname{Lem}{Lemma}{Lemmata}
\crefname{Kor}{Corollary}{Corollaries}
\crefname{Def}{Definition}{Definitions}

\let\div\relax

\newcommand{\N}{\mathbbm{N}} 
\newcommand{\Z}{\mathbbm{Z}} 
 
\newcommand{\R}{\mathbbm{R}}

\newcommand{\del}{\partial}
\newcommand{\lto}{\longrightarrow}

\newcommand{\blank}{\,\cdot\,}
\newcommand{\dvol}{\mathrm{dvol}}
\newcommand{\vol}{\mathrm{vol}}

\newcommand{\set}[2]{\left\{#1\,\middle| \,#2 \right\}}
\newcommand{\upd}{\mathrm{d}}
\newcommand{\Lie}{\mathcal{L}}
\newcommand{\dotcup}{\mathbin{\dot\cup}}

\renewcommand{\emptyset}{\varnothing}
\renewcommand{\epsilon}{\varepsilon}
\let\strphi\phi
\renewcommand{\phi}{\varphi}
\renewcommand{\hat}{\widehat}

\DeclareMathOperator{\grad}{grad}
\DeclareMathOperator{\div}{div}

\DeclareMathOperator{\im}{im}

\DeclareMathOperator{\tr}{tr}

\DeclareMathOperator{\supp}{supp}

\DeclareMathOperator{\scal}{scal}
\DeclareMathOperator{\ric}{ric}

\DeclareMathOperator{\Ein}{Ein}
\DeclareMathOperator{\Graph}{Graph}
\DeclareMathOperator{\Int}{int}
\DeclareMathOperator{\Clos}{cl}

\newcommand{\cf}{cf.~}
\newcommand{\ie}{i.\,e.~}
\newcommand{\eg}{e.\,g.~}
\begin{document}
\author{Jonathan Glöckle\thanks{Fakultät für Mathematik, Universität Regensburg, 93040 Regensburg, Germany, E-mail address: \href{mailto:jonathan.gloeckle.math@outlook.de}{jonathan.gloeckle.math@outlook.de}}}
\title{Initial data rigidity implies spacetime rigidity}
\date{\today}
\maketitle

\begin{abstract}
	In this article, we revisit the initial data rigidity theorem of Eichmair, Galloway and Mendes \cite{Eichmair.Galloway.Mendes:2021}.
	The goal is to strengthen their result by showing that the initial data sets concerned carry a vector field that is lightlike and parallel in an ambient sense.
	This will be used in a second step to show that among the spacetimes satisfying the dominant energy condition there exists locally essentially one spacetime extending these initial data sets.
	This local uniqueness theorem also applies in the context of other initial data rigidity theorems.
	Notably, the one in the spin case due the author \cite{Gloeckle:2023p} and a recent study of the mass zero case in the positive energy theorem due to Hirsch and Zhang \cite{Hirsch.Zhang:2024p}.
\end{abstract}

\section{Introduction}
The solution of the Cauchy problem for the vacuum Einstein equations due to Yvonne Choquet-Bruhat \cite{Foures-Bruhat:1952} was undoubtedly a major milestone in the mathematical investigation of general relativity.
It states the following.
Let $M$ be a smooth $n$-manifold, $n \geq 2$, equipped with a (smooth) \emph{initial data set} $(g,k)$, \ie a pair of a Riemannian metric $g$ and a symmetric $2$-tensor $k$ on $M$.
Assume that $(g,k)$ satisfies the \emph{vacuum constraints} $\rho = 0$ and $j = 0$ for
\begin{equation} \label{eq:CE}
	\begin{aligned}
		\rho &= \frac{1}{2}(\scal^g + \tr^g(k)^2 - |k|_g^2) \\
		j &= \div^g(k) - \upd \tr^g(k).
	\end{aligned}
\end{equation}
Then there is a globally hyperbolic Lorentzian manifold $(\overline{M}, \overline{g})$ subject to the vacuum Einstein equation $\Ein^{\overline{g}} = 0$ and an embedding $M \hookrightarrow \overline{M}$ as spacelike hypersurface such that $g$ is the induced metric and $k$ is the induced second fundamental form of $M$ in $(\overline{M},\overline{g})$.
Moreover, this Lorentzian manifold is \emph{locally geometrically unique}, meaning that for any two such Lorentzian manifolds $(\overline{M}_1, \overline{g}_1)$ and $(\overline{M}_2, \overline{g}_2)$ there are open neighborhoods $W_1 \subseteq \overline{M}_1$ and $W_2 \subseteq \overline{M}_2$ of the image of $M$ such that there is an time-orientation preserving isometric diffeomorphism $W_1 \to W_2$ fixing $M$.

When matter is involved, it gets considerably more difficult.
While there are similar local existence and uniqueness results for certain specific matter models, \eg Einstein-Maxwell theory \cite[Thm.~10.3,Cor.~10.4]{Choquet-Bruhat:2009} or Einstein-Dirac-Maxwell theory \cite[Lem.~3.20, 3.21]{Mueller.Nowaczyk:2017}, such a treatment is well out of reach for the plethora of matter filling our universe (for instance, a certain limit would otherwise give rise to solutions of the Navier-Stokes equation).
Cosmological considerations like the celebrated incompleteness theorems by Hawking and Penrose (\cf \cite[Sec.~8.2]{Hawking.Ellis:1973} or \cite[Thms.~14.55,14.61]{ONeill:1983}) therefore rely on certain inequalities -- known as energy conditions -- rather than the precise form of the Einstein field equations with matter.
The most prominent of those is the \emph{dominant energy condition} (=DEC).
It states that $\Ein^{\overline{g}}(X,Y) \geq 0$ for each pair of causal vectors $X, Y$ lying in the same half of the light cone and reflects the physical intuition that matter has non-negative energy and should not propagate faster than light.

In the present article we want to study the “Cauchy problem” for the dominant energy condition.
Noting that the induced initial data set $(g,k)$ on any spacelike hypersurface within a time-oriented DEC Lorentzian manifold satisfies $\rho \geq |j|_g$ (defined by \eqref{eq:CE}), the \emph{dominant energy condition for initial data sets}, we ask the following.

\begin{Frag}[Existence question]
	Given a DEC initial data set $(g,k)$ on a manifold $M$, is there a time-oriented Lorentzian manifold $(\overline{M}, \overline{g})$ subject to DEC and an embedding $M \hookrightarrow \overline{M}$ as spacelike hypersurface such that $(g,k)$ is the induced initial data set on $M$?
\end{Frag}

\begin{Frag}[Uniqueness question] \label{Que:Uniq}
	Are there situations, where such an $(\overline{M}, \overline{g})$ is locally geometrically unique? 
\end{Frag}

For simplicity, we define a \emph{spacetime} to be a time-oriented Lorentzian manifold.
We call a spacetime $(\overline{M}, \overline{g})$ an \emph{extension} of an initial data set $(M,g,k)$ if $M$ embeds into $(\overline{M}, \overline{g})$ as spacelike hypersurface with induced initial data set $(g,k)$ on $M$.
Throughout this article, all structures are considered to be smooth and unless explicitly stated the manifolds do not have boundary. 

In many cases, the answer to the existence question is yes.
This applies for instance to the vacuum case $\rho = 0, j = 0$ as a consequence of the solution of the Cauchy problem for the vacuum Einstein equations.
Similarly, if the initial data set $(g,k)$ can be complemented by initial data for an electromagnetic (and a Dirac) field such that the relevant constraints are satisfied, then the solution of the Cauchy problem for Einstein-(Dirac-)Maxwell theory gives a DEC spacetime with the required properties.
The same argument of course also applies to other matter models with solved Cauchy problem.
Furthermore, a simple Taylor expansion ansatz to second order provides a DEC spacetime extension whenever the strict DEC $\rho > |j|_g$ holds, \cf \cite[Prop.~1.2.13]{Gloeckle:2024_c} or \cite[Prop.~1.10]{Gloeckle:2019}.
Somewhat surprisingly, however, a (smooth) DEC spacetime extension does not exist in general.
An explicit example of a DEC initial data set that does not embed into a smooth DEC spacetime was constructed in~\cite{Gloeckle:2024_b}.

The example just mentioned is tightly linked to the uniqueness question.
One of the central observations in the article is \cite[Lem.~3]{Gloeckle:2024_b}, which states that if $\rho = |j|_g$ and DEC holds, then not only the $2n+1$ components of the Einstein curvature tensor that are directly given by $\rho$ and (twice) $j$ but all of its $(n+1)^2$ components are uniquely determined.
This at least indicates that some form of geometric uniqueness could hold in the case $\rho = |j|_g$, which is in stark contrast to the situation for $\rho > |j|_g$, where, for example, terms of order higher than two can be added at will.
A first uniqueness result can be derived in the vacuum case with the help of the so-called conservation theorem.
This seemingly little-known result in the book of Hawking and Ellis \cite[Sec.~4.3]{Hawking.Ellis:1973} states that if the DEC holds and the Einstein curvature tensor vanishes on some sufficiently regular subset (\eg a spacelike hypersurface) then it vanishes on the whole domain of dependence of that subset.
Consequently, if $\rho = |j|_g = 0$ for $(g,k)$ on $M$, then the domain of dependence of $M$ in $(\overline{M}, \overline{g})$ satisfies the vacuum Einstein equations and is thus locally geometrically uniquely determined by the solution of the Cauchy problem for the vacuum Einstein equations.

The main purpose of this article is to give a further answer to \cref{Que:Uniq}.
We will show a similar uniqueness result in the context of \emph{pp-wave spacetimes}, \ie Lorentzian manifolds $(\overline{M}, \overline{g})$ admitting a parallel lightlike vector field $V$.
Those are always time-orientable and we can arrange that $V$ is future-pointing.
We can restrict this vector~field to a spacelike hypersurface $M$ and, abusing notation, denote the result again by $V$.
It is a section of the vector bundle $T\overline{M}_{|M} \to M$.
This bundle can be canonically identified with $\overline{T}M \coloneqq \underline{\R} \oplus TM \to M$ via $(x,X) \mapsto xe_0 + X$, where $e_0$ denotes the future unit normal on $M$.
We will thus from now on write elements of $\overline{T}M$ as $xe_0 + X$ with $x \in \R$ and $X \in TM$.
This identification respects the fiberwise Lorentzian metric and the time-orientation if we define them on $\overline{T}M$ by $\overline{g}(xe_0 + X, x^\prime e_0 + X^\prime) = -xx^\prime + g(X, X^\prime)$ and future-pointedness of $e_0$, respectively.
There is a connection $\overline{\nabla}$ defined on $\overline{T}M \to M$ by
\begin{gather} \label{eq:OlConn}
	\overline{\nabla}_Y (xe_0 + X) = \Big(\del_Y x + k(Y,X) \Big) e_0 + \Big(xk(Y,\blank)^\sharp + \nabla_Y X\Big),
\end{gather}
with $Y \in TM$ and $xe_0 + X \in \Gamma(\overline{T}M)$.
The definition is made in such a way that under the identification $\overline{T}M \cong T\overline{M}_{|M}$ it agrees with the restriction of the Levi-Civita connection of $(\overline{M}, \overline{g})$ to $T\overline{M}_{|M} \to M$.
Hence it makes sense to talk of $V \in \Gamma(\overline{T}M)$ to be a future-lightlike $\overline{\nabla}$-parallel vector field if we just know the initial data set $(g,k)$ on $M$ and not the spacetime into which it embeds.

\begin{SATZ} \label{Thm:RigidSpacetime}
	Let $(g,k)$ be an initial data set on a manifold $M$ that carries a future-lightlike $\overline{\nabla}$-parallel vector field $V = ue_0 - U \in \Gamma(\overline{T}M)$, $U \in \Gamma(TM)$.
	Suppose that the triple $(g,k,V)$ is \emph{locally rigid} in the following sense:
	\begin{quote}
		\emph{For every $p \in M$ and any open neighborhood $\tilde{W}$ of $p$ there is an open neighborhood $W \subseteq \tilde{W}$ of $p$ such that every DEC initial data set $(g^\prime, k^\prime)$ coinciding with $(g,k)$ on $M \setminus W$ carries a future-lightlike $\overline{\nabla}$-parallel vector field $V^\prime$ coinciding with $V$ on $M \setminus W$.}
	\end{quote}
	Assume now that $(\overline{M}, \overline{g})$ is a DEC spacetime extension of $(M,g,k)$.
	Then there is an open neighborhood of $M$ in $\overline{M}$ admitting an isometric embedding of codimension $0$ into the \emph{Killing development}
	\begin{gather} \label{eq:KD}
		(\R \times M, -U^\flat \otimes \upd v - \upd v \otimes U^\flat + g),
	\end{gather}
	mapping $M$ to $\{0\} \times M$ in the canonical way and the future of $M$ into $\R_{>0} \times M$.
\end{SATZ}

The notion of Killing development goes back to Beig and Chruściel \cite[Sec.~II]{Beig.Chrusciel:1996}.
The construction produces a Lorentzian manifold with a Killing vector field from an initial data set equipped with a vector field $V \in \Gamma(\overline{T}M)$ that is Killing in a certain ambient sense.
In particular, it can be applied to a lightlike $\overline{\nabla}$-parallel vector field $V$ and yields the (up to isometry) unique Lorentzian manifold with complete parallel vector field $\frac{\del}{\del v}$ that contains $M$ as spacelike hypersurface intersected by all flow lines of $\frac{\del}{\del v}$ with induced initial data set $(g,k)$ and such that ${\frac{\del}{\del v}}_{|M} = V$.
\Cref{Thm:RigidSpacetime} gives a sufficient criterion under which the Killing development is also locally geometrically unique as DEC spacetime extension of $(M,g,k)$.

One might wonder whether the local rigidity assumption is satisfied in any example.
It is; in fact, \cref{Thm:RigidSpacetime} was developed in the author's PhD thesis (\cf \cite[Prop.~1.5.3]{Gloeckle:2024_c}) in order to obtain an application for the following rigidity theorem.
It should be noted that examples of initial data sets with the required properties can be constructed using a method developed by Ammann, Kröncke and Müller \cite{Ammann.Kroencke.Mueller:2021}, which is intended to be refined in upcoming work by Ammann, Kröncke and the author of this article \cite{Ammann.Gloeckle.Kroencke:2025pp}.

\begin{Satz}[{\cite[Main Theorem and Addendum]{Gloeckle:2023p}}] \label{Thm:RigidSpin}
	Let $M$ be a compact connected spin manifold with boundary $\del M = \del_+ M \dotcup \del_- M$ endowed with an initial data set $(g,k)$.
	Denote by $\nu$ the unit normal on $\del M$ that is inward-pointing along $\del_+ M$ and outward-pointing along $\del_- M$.
	Assume that
	\begin{itemize}
		\item $(g,k)$ satisfies the dominant energy condition $\rho \geq |j|_g$.
		\item The future null expansion scalar (with respect to $\nu$) $\theta^+ \coloneqq \tr^{\del M}(\nabla \nu) + \tr^{\del M}(k)$ satisfies $\theta^+ \leq 0$ on $\del_+ M$ and $\theta^+ \geq 0$ on $\del_- M$.
		\item The $\hat{A}$-genus of $\del_+ M$ is non-zero: $\hat{A}(\del_+ M) \neq 0$.
	\end{itemize}
	Then there is a diffeomorphism $\Phi \colon [0,\ell] \times \del_- M \to M$ defining a foliation $F_s = \Phi(\{s\} \times \del_- M)$ with $F_0 = \del_+ M$ and $F_\ell = \del_- M$.
	The leaves carry an induced metric, an induced spin structure and a unit normal $\nu$ pointing in the direction of growing $s$-parameter.
	The diffeomorphism can be chosen in such a way that the following holds for every leaf $F_s$:
	\begin{itemize}
		\item Its future null second fundamental form (with respect to $\nu$) $\chi^+ \coloneqq (\nabla \nu)^\flat_{|TF_s \otimes TF_s} + k_{|TF_s \otimes TF_s}$ vanishes, in particular it is a MOTS, \ie $\theta^+ = \tr^{F_s}(\chi^+)=0$.
		\item It carries a non-trivial parallel spinor, in particular its metric is Ricci-flat.
		\item Its tangent vectors are orthogonal to $j^\sharp$ and $\rho + j(\nu) = 0$, in particular the dominant energy condition holds marginally: $\rho = |j|_g$.
	\end{itemize}
	Moreover, the initial data set $(g,k)$ on $M$ carries a lightlike $\overline{\nabla}$-parallel spinor $\phi$.
	The foliation $(F_s)_{s \in [0,\ell]}$ may be constructed in such a way that the Riemannian Dirac current $U_\phi$ of $\phi$ is a positive multiple of $-\nu$ at each point.
\end{Satz}

\begin{Bem}
	The condition on $\theta^+$ says that $\del_+ M$ is weakly outer trapped and $\del_- M$ is weakly outer untrapped, where the outgoing direction points from $\del_+M$ to $\del_- M$.
	From a Riemannian perspective, it is a lower bound on the mean curvature $H = \frac{1}{n-1} \tr^{\del M}(-\nabla \tilde{\nu})$ of the boundary with respect to the inward-pointing unit normal $\tilde{\nu}$.
	Namely, $H \geq \pm (n-1)\tr^{\del M}(k)$ on $\del_{\pm} M$.
	Since, moreover, we assume a lower bound on the scalar curvature (implied by $\rho \geq |j|_g$) and an index theoretic obstruction to positive scalar curvature on $\del_+ M$ (given by $\hat{A}(\del_+ M) \neq 0$), \cref{Thm:RigidSpin} is very much in spirit of the warped product rigidity results due to Cecchini and Zeidler \cite[Sec.~8-10]{Cecchini.Zeidler:2024} and Daniel Räde \cite{Raede:2023}.
	The relation to these results is discussed more thoroughly in the introduction of \cite{Gloeckle:2023p}.
\end{Bem}

Actually, most of the conclusions in~\cref{Thm:RigidSpin} follow from the final addendum, which is proved first.
Without going into the details of spin geometry, we sketch how the foliation is constructed from the existence of the lightlike $\overline{\nabla}$-parallel spinor $\phi$ (with the right boundary conditions).
The full proof is laid out in \cite{Gloeckle:2023p}.
The main idea is considering the (Lorentzian) Dirac current $V_\phi = u_\phi e_0 - U_\phi \in \Gamma(\overline{T}M)$ associated to the spinor $\phi$ (\cf \eg \cite[Def.~20]{Gloeckle:2023p}).
The precise definition is not important here; it is merely the following fact that matters: $V_\phi$ is a future-lightlike (in particular nowhere vanishing) $\overline{\nabla}$-parallel vector field.
This amounts to $u_\phi = |U_\phi| > 0$ and
\begin{equation} \label{eq:OlPar}
	\begin{aligned}
		\nabla_X U_\phi &= u_\phi k(X, \blank)^\sharp \\
		\upd u_\phi(X) &= k(U_\phi, X),
	\end{aligned}
\end{equation}
for all $X \in TM$.
Thus the tangential part $U_\phi$, which goes by the name \emph{Riemannian Dirac current}, is nowhere vanishing and the first equation of~\eqref{eq:OlPar} shows that $\upd U_\phi^\flat = 0$.
So in particular, $U_\phi^\flat$ defines an involutive distribution.
The foliation $(F_s)_{s \in [0,\ell]}$ is now obtained from this distribution by virtue of the Frobenius theorem and $\overline{\nabla} V_\phi = 0$ immediately implies $\chi^+ = 0$ for the leaves.

It is important to note that the existence of $V_\phi$ is independent of how the initial data set $(g,k)$ looks like in the interior $M_0 \coloneqq M \setminus \del M$ -- as long as the DEC is satisfied.
Basically, this observation implies the local rigidity of $(g_{|M_0},k_{|M_0},{V_{\phi}}_{|M_0})$ and gives rise to the following corollary.

\begin{Kor} \label{Cor:LocGeomUniqSpin}
	Let $(g,k)$ be an initial data set on a manifold with boundary $M$ and suppose that the assumptions of \cref{Thm:RigidSpin} are satisfied.
	Then there is a DEC spacetime extension of the initial data set $(g_{|M_0}, k_{|M_0})$ on $M_0 \coloneqq M \setminus \del M$ and this extension is locally geometrically unique:
	For any two time-oriented DEC Lorentzian manifolds $(\overline{M}_1, \overline{g}_1)$ and $(\overline{M}_2, \overline{g}_2)$ containing $(M_0,g_{|M_0},k_{|M_0})$ as spacelike hypersurface there are open neighborhoods $W_1 \subseteq \overline{M}_1$ and $W_2 \subseteq \overline{M}_2$ of $M_0$ and a time-orientation preserving isometry $(W_1, \overline{g}_1) \cong (W_2, \overline{g}_2)$ fixing $M_0$.  
\end{Kor}

\Cref{Thm:RigidSpin} has a predecessor due to Eichmair, Galloway and Mendes that was proved with the help of minimal surface (or rather:~MOTS) techniques instead of spinorial methods.

\begin{Satz}[{\cite[Thm.~1.2]{Eichmair.Galloway.Mendes:2021}}] \label{Thm:RigidMOTS}
	Let $M$ be a compact connected manifold of dimension $3 \leq n \leq 7$ with boundary $\del M = \del_+ M \dotcup \del_- M$ endowed with an initial data set $(g,k)$.
	Denote by $\nu$ the unit normal on $\del M$ that is inward-pointing along $\del_+ M$ and outward-pointing along $\del_- M$.
	Assume that
	\begin{itemize}
		\item $(g,k)$ satisfies the dominant energy condition $\rho \geq |j|_g$.
		\item The future null expansion scalar (with respect to $\nu$) satisfies $\theta^+ \leq 0$ on $\del_+ M$ and $\theta^+ \geq 0$ on $\del_- M$.
		\item $M$ satisfies the \emph{homotopy condition} with respect to $\del_+ M$:
		There is a continuous map $M \to \del_+ M$ so that the composition $\del_+ M \hookrightarrow M \to \del_+ M$ is homotopic to the identity.
		\item $\del_+ M$ satisfies the \emph{cohomology condition}:
		There are cohomology classes $h_1, \ldots, h_{n-1} \in H^1(\del_+ M, \Z)$ with $h_1 \cup \ldots \cup h_{n-1} \neq 0$.
	\end{itemize}
	Then there is a diffeomorphism $\Phi \colon  [0,\ell] \times \del_- M \to M$ defining a foliation $F_s = \Phi(\{s\} \times \del_- M)$ with $F_0 = \del_+ M$ and $F_\ell = \del_- M$.
	The leaves carry an induced metric and a unit normal $\nu$ pointing in the direction of growing $s$-parameter.
	The diffeomorphism can be chosen in such a way that the following holds for every leaf $F_s$:
	\begin{itemize}
		\item Its future null second fundamental form (with respect to $\nu$) vanishes, $\chi^+ = 0$, in particular it is a MOTS.
		\item Its metric is Ricci-flat. (It follows, by an argument of Bochner, that $F_s$ with its induced metric is a flat torus.)
		\item Its tangent vectors are orthogonal to $j^\sharp$ and $\rho + j(\nu) = 0$, in particular the dominant energy condition holds marginally: $\rho = |j|_g$.
	\end{itemize}
\end{Satz}

It arises the question, whether also in the context of \cref{Thm:RigidMOTS}, there is a future-lightlike $\overline{\nabla}$-parallel vector field inducing the foliation so that \cref{Thm:RigidSpacetime} can be applied.
While $\chi^+ = 0$ is equivalent to $\overline{g}(\overline{\nabla}_X (e_0 + \nu), Y) = 0$ for $X,Y \in TF_s$, it is not clear why the other components of $\overline{\nabla} (e_0 + \nu)$ should vanish -- and in general they do not.
Instead, we use the ansatz $V = u(e_0 + \nu)$ for a positive function $u \in C^\infty(M)$ and remark that $\overline{g}(\overline{\nabla}_X V, Y) = 0$ also holds for $X,Y \in TF_s$.
It is the second main result of this article that in the context of \cref{Thm:RigidMOTS}, this function can be chosen appropriately.
This leads to a strengthening of the initial data rigidity result that is interesting in its own right.

\begin{SATZ} \label{Thm:ExParVF}
	Let $(g,k)$ be a DEC initial data set on a compact manifold with boundary $M$.
	Assume that $M \cong [0,\ell] \times F$ in such a way that all the canonical leaves are Ricci-flat MOTS.
	Then there exists a lightlike $\overline{\nabla}$-parallel vector field on $M$.
	More precisely, there is a positive function $u \in C^\infty(M)$ such that $V \coloneqq ue_0 + u\nu$ satisfies $\overline{\nabla} V = 0$.
\end{SATZ}

As a consequence, \cref{Thm:RigidSpacetime} can also be applied in the situation of \cref{Thm:RigidMOTS}, yielding local geometric uniqueness for the DEC spacetime extension.
It is worth noting that -- as verified in \cref{Prop:KDSatDEC} later -- in the setting of \cref{Thm:ExParVF}, the Killing development is subject to DEC yielding the asserted DEC spacetime extension both in Corollary \labelcref{Cor:LocGeomUniqSpin} and \labelcref{Cor:LocGeomUniqMOTS}. 

\begin{Kor} \label{Cor:LocGeomUniqMOTS}
	Let $(g,k)$ be an initial data set on a manifold with boundary $M$ and suppose that the assumptions of \cref{Thm:RigidMOTS} are satisfied.
	Then there is a DEC spacetime extension of the initial data set $(g_{|M_0}, k_{|M_0})$ on $M_0 \coloneqq M \setminus \del M$ and this extension is locally geometrically unique. 
\end{Kor}

Lastly, let us also note that there is a third situation where \cref{Thm:RigidSpacetime} can be applied.
Namely, Hirsch and Zhang recently showed that asymptotic conditions can also give rise to a lightlike $\overline{\nabla}$-parallel vector field.

\begin{Satz}[{\cite[Thm.~1.1]{Hirsch.Zhang:2024p}}] \label{Thm:RigidAsympt}
	Let $(M, g, k)$ be a (connected) $C^{2,\alpha}$-asymptotically flat DEC initial data set with decay rate $q \in (\frac{n-2}{2}, n-2]$.
	Suppose that $M$ is spin and that $E = |P|$ for some end.
	Then $M$ embeds into the pp-wave spacetime $(\R^{n+1},\overline{g})$ with 
	\begin{gather*}
		\overline{g} = -\upd s \otimes \upd v - \upd v \otimes \upd s + f \upd s^2 + g_{\R^{n-1}}
	\end{gather*}
	for some function $f \in C^\infty(\R \times \R^{n-1})$, viewed as function on $\R^{n+1}$ not depending on $v$, that is superharmonic as function on $F_s = \{s\} \times \R^{n-1}$ for all $s \in \R$.
	Moreover, $M$ is diffeomorphic to $\R^n$ (\ie has only one end in particular) and if $E = |P| > 0$, then the spacetime above may be chosen such that $\frac{\del}{\del v}_{|M}$ is asymptotic to $E e_0 + P$.
\end{Satz}

The spacetime $(\R^{n+1}, \overline{g})$ is indeed a pp-wave spacetime in our sense since the lightlike coordinate vector field $\frac{\del}{\del v}$ is ($\overline{\nabla}$-)parallel, which follows directly from the Koszul formula.
Further straightforward calculation shows that $\Ein^{\overline{g}} = \frac{1}{2}(\triangle^{F_s} f)\frac{\del}{\del v}^\flat \otimes \frac{\del}{\del v}^\flat$, which implies that the spacetime is DEC for superharmonic $f$.
The final sentence is not part of the original theorem but follows directly from its proof as we will remark in the proof of \cref{Cor:LocGeomUniqAsympt}.
Even though \Cref{Thm:RigidAsympt} states that $(M, g, k)$ embeds into a pp-wave spacetime that satisfies the DEC, it does not exclude the possibility of a different DEC spacetime extension of $(M, g, k)$.
This follows from \cref{Thm:RigidSpacetime} as \cref{Thm:RigidAsympt} applies to any DEC initial data set resulting form a change of $(g,k)$ in a bounded region.
In particular, we obtain local geometric uniqueness again.

\begin{Kor} \label{Cor:LocGeomUniqAsympt}
	Let $(g,k)$ be an asymptotically flat initial data set on a manifold $M$ and suppose that the assumptions of \cref{Thm:RigidAsympt} are satisfied.
	Then there is a DEC spacetime extension of the initial data set $(M, g, k)$ and this extension is locally geometrically unique. 
\end{Kor}

The main part of the article consists of two sections that can be read independently.
The first one, \ie \cref{Sec:ExParVF}, is devoted to the proof of \cref{Thm:ExParVF} improving Eichmair, Galloway and Mendes' initial data rigidity result.
Here, we first analyze the variation formula for the null second fundamental form and observe that in the rigid situation, there is a term whose vanishing has not been exploited in \cite{Eichmair.Galloway.Mendes:2021}.
Although this vanishing is not entirely new -- it appears for instance as \cite[eq.~(3.26)]{Galloway.Mendes:2018} -- it seems to be a novel observation that as consequence there is a lightlike $\overline{\nabla}$-parallel vector field along the leaves.
We then remark that the existence of a lightlike $\overline{\nabla}$-parallel vector field on all of $M$ is equivalent to the vanishing of a single component of the curvature tensor.
With the help of the deformation theory for Ricci-flat metrics we will be able to conclude that this term is zero.

In the second section, \ie \cref{Sec:RigidSpacetime}, we give the proof of \cref{Thm:RigidSpacetime}, our spacetime rigidity theorem.
The rather straightforward, yet technical, proof of the theorem is contained in the author's thesis \cite{Gloeckle:2024_c} but has so far not been part of a journal article.
After this proof, we carefully prove the corollaries, the ones using the non-spinorial and the asymptotically flat initial data rigidity results being entirely new.

\subsection*{Acknowledgements}
I am grateful to Eduardo Hafemann for patiently explaining to me the general logic of the MOTS existence results.
Moreover, I want to thank Greg Galloway for discussions on the topic at the SLMath, Berkeley, and further guidance through the MOTS literature.
I am also thankful to Bernd Ammann for all of his support.
During different stages of the project, I was funded by the DFG through the SFB 1085 “Higher Invariants” in Regensburg and the SPP 2026 “Geometry at Infinity”.

\section{Constructing lightlike parallel vector fields} \label{Sec:ExParVF}
The goal of this section is to prove \cref{Thm:ExParVF} \ie to construct a lightlike $\overline{\nabla}$-parallel vector field in the setting of \cref{Thm:RigidMOTS}.
We complete this in two steps.
First, we construct such a vector field along each of the leaves foliating the manifold.
Second, we show that an appropriate curvature term vanishes so that a leafwise constant rescaling of these vector fields produces a vector field that is $\overline{\nabla}$-parallel everywhere.

Throughout this section, let $M$ be a manifold with boundary and $(g,k)$ an initial data set on $M$.
Suppose that there is a diffeomorphism $M \cong [0,\ell] \times F$ with $F$ a closed manifold, $\ell \in \R_{>0}$, and denote by $s \colon M \to [0,\ell]$ the projection on the second component.
We assume furthermore that each of the canonical leaves $F_{\tau} = t^{-1}(\tau)$ for $\tau \in [0,\ell]$ satisfies $\chi^+ = 0$ with respect to the unit normal $\nu$ pointing in the direction of $\grad(s)$.
Let $\strphi \coloneqq \upd s(\nu)^{-1}$.
The flow of the vector field $\strphi \nu$ maps leaves to leaves.
It hence gives rise to a diffeomorphism $M \cong [0,\ell] \times F$ with the same canonical foliation and the additional property that the vector field $\frac{\del}{\del s}$ that corresponds to differentiating the $[0,\ell]$-coordinate is equal to $\strphi \nu$ and thus orthogonal to the leaves.
We may thus assume that we work in the following setup:

\begin{Set} \label{Set:Prod}
	Let $M = [0,\ell] \times F$ with $F$ a closed manifold be endowed with an initial data set $(g,k)$ such that
	\begin{itemize}
		\item $g= \strphi^{2} \upd s^2 + g_s$, where $(g_s)_{s \in [0,\ell]}$ is a family of metrics on $F$ and
		\item all leaves $F_\tau$ with $\tau \in [0,\ell]$ are MOTS.
	\end{itemize}
\end{Set}

\begin{Prop} \label{Prop:AddRigid}
	Assume \cref{Set:Prod} and additionally that the DEC holds and that all the metrics $(g_s)_{s \in [0,\ell]}$ on the leaves are scalar-flat.
	Then the positive function $\strphi = \upd s(\nu)^{-1} \in C^\infty(M)$ satisfies 
	\begin{gather*}
		\upd \log \strphi(X) = k(X,\nu)
	\end{gather*}
	for all $X$ tangent to the leaves $F_\tau$.
	Moreover, the DEC holds marginally in the precise sense $\rho + j(\nu) = 0$ and all the leaves satisfy $\chi^+ = 0$.
\end{Prop}
\begin{Kor} \label{Cor:LeafPar}
	Assume \cref{Set:Prod} and additionally that the DEC holds and that all the metrics $(g_s)_{s \in [0,\ell]}$ on the leaves are scalar-flat.
	Then the vector field $V \coloneqq ue_0 + u\nu$ with $u = \strphi^{-1}$ satisfies $\overline{\nabla}_X V = 0$ for all $X$ tangent to the leaves $F_\tau$.
\end{Kor}
\begin{proof}
	Note first that the connection $\overline{\nabla}$ defined by \eqref{eq:OlConn} is metric with respect to the Lorentzian metric $\overline{g}$ on $\overline{T}M$.
	This follows either by direct calculation or from the fact that $\overline{\nabla}$ and $\overline{g}$ are the restriction to $M$ of the Levi-Civita connection and the Lorentzian metric, respectively, of any Lorentzian manifold into which $(M,g,k)$ embeds.
	Thus we may calculate for any $u \in C^\infty(M)$:
	\begin{equation} \label{eq:ParVsK}
	\begin{aligned}
		\overline{g}(\overline{\nabla}_X V, Y) &= (\del_X u) \overline{g}(e_0+\nu, Y) + u \overline{g}(\overline{\nabla}_X(e_0+\nu), Y) = u \chi^+(X,Y) \\
		\overline{g}(\overline{\nabla}_X V, e_0 + \nu) &=  (\del_X u) \overline{g}(e_0+\nu, e_0+\nu) + u \overline{g}(\overline{\nabla}_X(e_0+\nu), e_0+ \nu) \\
			&= \frac{1}{2} u \del_X \overline{g}(e_0+\nu, e_0+\nu) = 0 \\
		\overline{g}(\overline{\nabla}_X V, \nu) &=  (\del_X u) \overline{g}(e_0+\nu, \nu) + u \overline{g}(\overline{\nabla}_X(e_0+\nu), \nu) \\
			&= \del_X u + u k(X,\nu) + \frac{1}{2} u \del_X \overline{g}(\nu, \nu) = \del_X u + u k(X,\nu),
	\end{aligned}
	\end{equation}
	where $X, Y \in TM$ are tangent to the leaves and $V = u(e_0 + \nu)$.
	Now the claim immediately follows by setting $u = \strphi^{-1}$ and invoking \cref{Prop:AddRigid}.
\end{proof}

\begin{proof}[Proof of \cref{Prop:AddRigid}]
	The proof is similar to the one of  \cite[Thm.~3.1]{Galloway:2008}, on which \cite[Thm.~1.2]{Eichmair.Galloway.Mendes:2021} relies, in that it consists of analyzing the first-order variation formula for the future null second fundamental form $\theta^+$.
	For any $\tau \in [0,\ell]$, we consider the variation of $F_\tau = F \times \{\tau\}$ by $(F_s)_{s \in [0,\ell]}$.
	Since the variation vector field is $\frac{\del}{\del s} = \phi \nu$ and $F_\tau$ is a MOTS the formula \cite[Prop~7.32]{Lee:2019}	reads
	\begin{gather*}
		\frac{\del \theta^+}{\del s}_{|s = \tau} = -\triangle^{F_\tau} \strphi + 2\del_X \strphi + (\div^{F_\tau} X - |X|^2 + Q_\tau) \strphi,
	\end{gather*}
	where
	\begin{gather*}
		Q_\tau = \frac{1}{2} \scal^{F_\tau} - \Bigl(\rho + j(\nu)\Bigr) - \frac{1}{2} \left|\chi^+\right|^2 \in C^\infty(F_\tau)
	\end{gather*}
	and $X = \pi^{\tan}(k(\nu,\blank)^\sharp) \in \Gamma(TF_\tau)$.
	Setting $Y \coloneqq X - \grad^{F_\tau}(\log \strphi)$ as in  \cite[Lem.~4.2.5]{Haferman:2023} (which is probably based on \cite[(2.7)-(2.10)]{Galloway.Schoen:2006}) this expression simplifies to
	\begin{gather} \label{eq:VarEqSimp}
		\frac{\del \theta^+}{\del s}_{|s = \tau} = (\div^{F_\tau} Y - |Y|^2 + Q_\tau) \strphi.
	\end{gather}
	Now the left-hand side of \eqref{eq:VarEqSimp} vanishes since all the leaves $F_t$ are MOTS.
	We divide \eqref{eq:VarEqSimp} by $-\strphi$, recall that $\scal^{F_\tau} = 0$, integrate over $F_\tau$ and obtain
	\begin{gather*}
		0 = \int_{F_\tau} \left(\left|Y\right|^2 + \Bigl(\rho + j(\nu)\Bigr) + \frac{1}{2} \left|\chi^+ \right|^2 \right)\dvol^{g_\tau}.
	\end{gather*}
	Since all summands on the right-hand side are non-negative -- the middle one due to the DEC -- all of them must vanish identically.
	In particular, this shows that $Y \equiv 0$, which is equivalent to $k(\nu, \blank)^\sharp - \grad(\log(\strphi))$ being normal to the leaf $F_\tau$.
\end{proof}

\begin{Bem}
	As the proof shows, it would also have been sufficient to only assume non-positivity of the scalar curvature instead of scalar-flatness in \cref{Prop:AddRigid}.
	But in the case of interest, the metrics $(g_s)_{s \in [0,\ell]}$ will be Ricci-flat anyhow and we shall also use this Ricci-flatness in the second part of the argument.
\end{Bem}

Since there is freedom to reparameterize the $s$-coordinate, which results in a leafwise scaling of $\strphi$, the definition of $V$ in terms of $\strphi$ cannot be expected to yield a $\overline{\nabla}$-parallel vector field on all of $M$ without properly fixing the parameterization.
That this can be done amounts to the vanishing of a certain curvature term as we shall see now.

\begin{proof}[Proof of \cref{Thm:ExParVF}]
	We start by showing that $\overline{\nabla}_Z (e_0+\nu)$ is a multiple of $e_0 + \nu$ for all $Z \in TM$.
	This is already clear for all $Z$ tangent to the leaves by \cref{Cor:LeafPar}.
	For $Z = \nu$, we first calculate $\nabla_{\nu} \nu$.
	To do so, let $X_F \in \Gamma(TF)$ be any vector field.
	The product structure $M = [0,\ell] \times F$ allows to define $X(s,p) \coloneqq (0,X_F(p)) \in T_{(s,p)}M$ for all $s \in [0,\ell]$, $p \in F$.
	The resulting vector field $X$ is tangential to all leaves and satisfies $[X, \frac{\del}{\del s}] = 0$.
	Recalling $\frac{\del}{\del s} = \strphi \nu$, this implies $[X, \nu] = - \upd \log \strphi(X) \cdot \nu$.
	Hence, with the help of \cref{Prop:AddRigid}, we get
	\begin{align} \label{eq:NablaNuNu}
		g(\nabla_{\nu} \nu, X) = -g(\nu, \nabla_{\nu} X) = g(\nu, [X, \nu]) = -\upd \log \strphi(X) = -k(X,\nu).
	\end{align}
	Since any vector tangent to a leaf can be extended to such a vector field $X$ and
	\begin{gather*}
		g(\nabla_{\nu} \nu, \nu) = \frac{1}{2} \del_{\nu} g(\nu, \nu) = 0,
	\end{gather*}
	we obtain $\nabla_{\nu} \nu = -\pi^{\tan}(k(\blank, \nu)^\sharp)= -k(\blank, \nu)^\sharp + k(\nu,\nu) \nu$.
	Thus we get
	\begin{gather*}
		\overline{\nabla}_{\nu} (e_0 + \nu) = k(\blank, \nu)^\sharp + k(\nu, \nu) e_0 + \nabla_{\nu} \nu = k(\nu,\nu) (e_0 + \nu).
	\end{gather*}
	
	Hence, there is a (positive) function $u$ such that $u(e_0 + \nu)$ is $\overline{\nabla}$-parallel if and only if the curvature $\overline{R}(Z,W) (e_0 + \nu) \coloneqq \overline{\nabla}_Z \overline{\nabla}_W (e_0 + \nu) - \overline{\nabla}_W \overline{\nabla}_Z (e_0 + \nu) - \overline{\nabla}_{[Z,W]} (e_0 + \nu)$ vanishes for all $Z,W \in \Gamma(TM)$.
	Clearly, at each point it is a multiple of $e_0 + \nu$ and thus it suffices now to show that $\overline{g}(\overline{R}(Z,W) (e_0 + \nu), \nu) = 0$ for all $Z,W \in TM$.
	By \cref{Cor:LeafPar}, $\overline{R}(X,Y) (e_0 + \nu) = 0$ if both $X$ and $Y$ are tangential to the leaves.
	So we just need show that $\overline{g}(\overline{R}(X,\nu) (e_0 + \nu), \nu) = 0$ for all $X \in TF_\tau$, $\tau \in [0,\ell]$, which will be done in \cref{Thm:RicciFlatDefSatJEq}.
	To apply it, note that the equation $\rho + j(\nu) = 0$ from \cref{Prop:AddRigid} together with the DEC implies $j = -\rho\nu^\flat$ (\cf\eqref{eq:MargDEC}) and in particular $j_{|TF_\tau} = 0$ for any leaf $F_\tau$.
\end{proof}

We now fix a leaf $F_{\tau}$ and define the $1$-form $\lambda \in \Omega^1(F_{\tau})$ by
\begin{gather} \label{eq:DefLambda}
	\lambda(X) \coloneqq \overline{g}(\overline{R}(X,\nu)(e_0+\nu),\nu) = \overline{g}(\overline{R}(X,\nu) e_0,\nu) = \nabla_X k(\nu,\nu) - \nabla_{\nu} k(X,\nu).
\end{gather}
By what has been said above and since the particular leaf was arbitrarily chosen, it suffices to show that $\lambda = 0$.

The following two lemmas are obtained straightforward calculations.
In the proofs, we will use the equations $\upd \log \strphi(X) = k(X,\nu)$ and $k(X,Y) = -g(\nabla_X \nu, Y) = g(\nabla_X Y, \nu)$ for $X, Y \in \Gamma(TF_\tau)$ that are, by~\eqref{eq:ParVsK}, equivalent to leafwise $\overline{\nabla}$-parallelism of $u(e_0 + \nu)$.
\begin{Lem} \label{Lem:2for3}
	In \cref{Set:Prod} and under the additional assumption that $u(e_0 + \nu)$ is leafwise $\overline{\nabla}$-parallel for $u = \strphi^{-1}$, the following holds for all $X \in TF_{\tau}$:
	\begin{gather*}
		j(X) + \lambda(X) = -\frac{1}{2}\strphi^{-1}\left(\div^{F_{\tau}}(\dot{g}_{\tau}) - \upd \tr^{F_{\tau}}(\dot{g}_{\tau})\right)(X),
	\end{gather*}
	where $\dot{g}_{\tau} = \frac{\upd}{\upd s}_{|s = \tau} g_s$ and $\lambda$ is defined by~\eqref{eq:DefLambda}.
\end{Lem}
\begin{proof}
	Recalling the definitions of $j$ and $\lambda$ as well as the Codazzi equation, we obtain
	\begin{equation} \label{eq:CodazziStep}
		\begin{aligned}
			j(X) + \lambda(X) &= \tr\left(\overline{g}(\overline{R}(\blank, X) e_0, \blank)\right) - \overline{g}(\overline{R}(\nu, X) e_0, \nu) \\
				&= \tr^{F_{\tau}}\left(\overline{g}(\overline{R}(\blank, X) e_0, \blank)\right) \\
				&= \tr^{F_{\tau}}\Bigl((\nabla k)(X,\blank) - (\nabla_X k)\Bigr)
		\end{aligned}
	\end{equation}
	for all $X \in TF_{\tau}$.
	Now, we note that
	\begin{align*}
		\dot{g}_{\tau}(X,Y) = (\Lie_{\frac{\del}{\del s}} g_t)(X,Y) &= g_t\left(\nabla_X \frac{\del}{\del s}, Y\right) + g_s \left(X, \nabla_Y \frac{\del}{\del s} \right) \\
			&= \strphi \Bigl(g_s(\nabla_X \nu, Y) + g_s(X, \nabla_Y \nu) \Bigr) = -2\strphi k(X,Y)
	\end{align*}
	for all $X,Y \in TF_{\tau}$.
	This allows to calculate for all $X, Y, Z \in \Gamma(TF_{\tau})$ that
	\begin{align*}
		2\strphi(\nabla_X k)(Y,Z) &= 2\strphi\del_X(k(Y,Z)) - 2\strphi k(\nabla_X Y, Z) - 2\strphi k(Y, \nabla_X Z) \\
			&= -\strphi \del_X(\strphi^{-1}\dot{g}_{\tau}(Y,Z)) + \dot{g}_{\tau}(\nabla^{F_{\tau}}_X Y, Z) -2\strphi g(\nabla_X Y, \nu)k(\nu, Z) \\
			&\phantom{=}\;+ \dot{g}_{\tau}(Y,\nabla^{F_{\tau}}_X Z) -2 \strphi g(\nabla_X Z, \nu)k(Y, \nu) \\
			&= \upd \log \strphi(X) \dot{g}_{\tau}(Y,Z) - (\nabla^{F_{\tau}}_X \dot{g}_{\tau})(Y,Z) \\
			&\phantom{=}\;+ \dot{g}_{\tau}(X,Y) \upd \log \strphi(Z) + \dot{g}_{\tau}(X,Z) \upd \log \strphi(Y)
	\end{align*}
	and thus
	\begin{gather*}
		(\nabla_X k)(Y,Z) - (\nabla_Y k)(X,Z) = -\frac{1}{2} \strphi^{-1} \left((\nabla^{F_{\tau}}_X \dot{g}_{\tau})(Y,Z) - (\nabla^{F_{\tau}}_Y \dot{g}_{\tau})(X,Z)\right).
	\end{gather*}
	The claimed equation now follows by tracing this over $Y$ and $Z$ and inserting the result into~\eqref{eq:CodazziStep}.
\end{proof}

\begin{Lem} \label{Lem:Closed}
	In \cref{Set:Prod}, we have
	\begin{gather*}
		\Omega^2(F_{\tau}) \ni \upd (\strphi \lambda) = 0
	\end{gather*}
	for $\lambda$ defined by~\eqref{eq:DefLambda}.
\end{Lem}
In a certain way that we make precise in \cref{Sec:AppBianchi}, this lemma follows from the second Bianchi identity for some spacetime extension of $(M,g,k)$.
The shorter proof given here works entirely on the level of initial data sets.
\begin{proof}
	Let $X,Y \in \Gamma(TM)$ with $X \perp \nu$ and $Y \perp \nu$ everywhere and $[X,\frac{\del}{\del s}] = 0 = [Y, \frac{\del}{\del s}]$.
	It was explained in the proof of~\cref{Thm:ExParVF} how to construct for any vector in $TF_\tau$ such a vector field extending it.
	Then, with the help of $ g(\nabla_X \nu, \nabla_\nu \nu) = - k(\nabla_X \nu, \nu)$, which follows from~\eqref{eq:NablaNuNu},
	\begin{align*}
			\lambda(X) &= \nabla_X k(\nu,\nu) - \nabla_\nu k(X,\nu) \\
				&= \del_X k(\nu,\nu) - 2 k(\nabla_X \nu,\nu) - \del_\nu k(X,\nu) + k(\nabla_{\nu} X,\nu) + k(X,\nabla_{\nu} \nu) \\
				&= \del_X k(\nu,\nu) - 2 k(\nabla_X \nu, \nu) - \del_\nu k(X,\nu) + k(\nabla_X \nu, \nu) + k([\nu, X],\nu) - g(\nabla_X \nu, \nabla_\nu \nu) \\
				&= \del_X k(\nu,\nu)  - \del_\nu k(X,\nu) + k(\upd \log \strphi(X)\nu,\nu) \\
				&= \del_X k(\nu,\nu)  - \del_\nu \del_X \log \strphi + (\del_X \log \strphi)k(\nu,\nu)
	\end{align*}
	holds on $F_{\tau}$.
	Thus, using $\del_X \del_\nu \del_Y \log \strphi = \del_\nu \del_X \del_Y \log \strphi + \del_{[X,\nu]} \del_Y \log \strphi$ with $[X,\nu] = - (\del_X \log \strphi) \nu$, we get
	\begin{align*}
		\del_X (\lambda(Y)) &= \del_X\del_Y k(\nu,\nu) - \del_\nu \del_X \del_Y \log \strphi  + (\del_X \log \strphi) \del_\nu \del_Y \log \strphi \\
		&\phantom{=}\,+ (\del_X \del_Y \log \strphi)k(\nu,\nu) + (\del_Y \log \strphi)\del_X k(\nu,\nu).
	\end{align*}
	The general relation $\del_X \del_Y f - \del_Y \del_X f - \del_{[X,Y]} f = 0$ for all $f \in C^\infty(M)$ now leads to multiple cancellation in the second step of the following chain of equations:
	\begin{align*}
		\upd \lambda (X,Y) &= \del_X (\lambda(Y)) - \del_Y (\lambda(X)) - \lambda([X,Y]) \\
			&= (\del_X \log \strphi) \del_\nu \del_Y \log \strphi + (\del_Y \log \strphi) \del_X k(\nu, \nu) \\
			&\phantom{=}\;- (\del_Y \log \strphi) \del_\nu \del_X \log \strphi - (\del_X \log \strphi) \del_Y k(\nu, \nu) \\
			&= -(\upd \log \strphi \wedge \lambda)(X,Y).
	\end{align*}
	This immediately implies
	\begin{gather*}
		\upd (\strphi\lambda) = \upd \strphi \wedge \lambda - \strphi (\strphi^{-1} \upd \strphi \wedge \lambda) = 0. \qedhere
	\end{gather*}
\end{proof}

\Cref{Lem:2for3} implies that 2-for-3 holds (\ie either two imply the third) for the following three equations: $j_{|TF_{\tau}} = 0$, $\lambda = 0$ and the \emph{j-equation} $\div^{F_{\tau}}(\dot{g}_{\tau}) - \upd \tr^{F_{\tau}}(\dot{g}_{\tau}) = 0$.
In our setting, we know that $j_{|TF_t} = 0$, but this is not yet sufficient to conclude $\lambda = 0$.
We will make use of the closedness of $\strphi \lambda$ and the deformation theory for Ricci-flat metrics to conclude that $\lambda$ has to vanish in the case of interest, thereby completing the proof of \cref{Thm:ExParVF}.
As a preparation, we prove the following lemma.

\begin{Lem} \label{Lem:DivPartHodge}
	Let $(F,g)$ be a Ricci-flat Riemannian manifold and $W \in \Gamma(TF)$.
	Then
	\begin{gather*}
		\div(\Lie_W g) - \upd \tr(\Lie_W g) = -\delta \upd W^\flat,
	\end{gather*}
	where $\delta$ denotes the formal adjoint of $\upd$.
\end{Lem}
\begin{proof}
	Throughout, let $X, Y \in \Gamma(TF)$ and $e_1, \ldots, e_{n-1}$ denote a local orthonormal frame of $F$.
	Then
	\begin{equation} \label{eq:Lie}
		\begin{aligned}
			(\Lie_W g)(X,Y) &= \del_W g(X,Y) - g([W, X],Y) -g(X,[W, Y]) \\
				&= g(\nabla_X W, Y) + g(X, \nabla_Y W) \\
				&= \nabla W^\flat(X,Y) + \nabla W^\flat(Y,X).
		\end{aligned}
	\end{equation}
	We thus get
	\begin{align*}
		(\div(\Lie_W g) - \upd \tr(\Lie_W g))(X) &= \sum_{i = 1}^{n-1} \left(\nabla^2 W^\flat(e_i,e_i,X) + \nabla^2 W^\flat(e_i,X,e_i) \right) \\
		&\phantom{=}\;- 2\sum_{i = 1}^{n-1} \nabla^2 W^\flat(X,e_i,e_i).
	\end{align*}
	On the other hand, since $\upd W^\flat(X,Y) = \nabla W^\flat(X,Y) - \nabla W^\flat(Y,X)$ and $-\delta \beta$, for $\beta \in \Omega^2(F)$, is the divergence of $\beta$ in the sense of a trace over the first two entries of $\nabla \beta$, we have
	\begin{align*}
		-\delta \upd W^\flat(X) &= \sum_{i = 1}^{n-1} \left(\nabla^2 W^\flat(e_i,e_i,X) - \nabla^2 W^\flat(e_i,X,e_i) \right).
	\end{align*}
	The difference is thus given by
	\begin{align*}
	(\div(\Lie_W g) - \upd \tr(\Lie_W g) + \delta \upd W^\flat)(X) &= 2\sum_{i = 1}^{n-1} \left(\nabla^2 W^\flat(e_i,X,e_i) - \nabla^2 W^\flat(X,e_i,e_i) \right) \\
		&= 2\ric(X,W) = 0. \qedhere
	\end{align*}
	
\end{proof}

\begin{Satz} \label{Thm:RicciFlatDefSatJEq}
	Assume \cref{Set:Prod} with the additional assumptions that $u(e_0 + \nu)$ is leafwise $\overline{\nabla}$-parallel for $u = \strphi^{-1}$, that all metrics $(g_s)_{s \in [0,\ell]}$ are Ricci-flat and that $j_{|TF_{\tau}} = 0$.
	Then $\lambda = 0$ for $\lambda$ defined by~\eqref{eq:DefLambda}.
\end{Satz}
\begin{proof}
	Combining \cref{Lem:2for3} and \cref{Lem:Closed} in the case $j_{|TF_{\tau}} = 0$ yields
	\begin{gather} \label{eq:ClosedCond}
		0 = \upd (\div^{F_{\tau}}(\dot{g}_{\tau}) - \upd \tr^{F_{\tau}}(\dot{g}_{\tau})),
	\end{gather}
	while we have to show that $\div^{F_{\tau}}(\dot{g}_{\tau}) - \upd \tr^{F_{\tau}}(\dot{g}_{\tau}) = 0$.
	We analyze these equations by decomposing the symmetric $2$-tensor $\dot{g}_{\tau}$.
	First, we subtract from $\dot{g}_{\tau}$ the term $cg_{\tau}$ with $c \coloneqq \frac{1}{(n-1)\vol(F_\tau,g_\tau)} \int_{F_\tau} \tr^{g_{\tau}}(\dot{g}_{\tau}) \dvol^{g_{\tau}}$, so that the mean value of the trace of $\dot{g}_{\tau} - cg_{\tau}$ vanishes.
	Secondly, we may choose a vector field $W \in \Gamma(F_\tau)$ such that $\div^{g_{\tau}}(\Lie_W g_{\tau}) = \div^{g_{\tau}}(\dot{g}_{\tau} - cg_{\tau})$.
	The difference term $h \coloneqq  \dot{g}_{\tau} - cg_{\tau} - \Lie_W g_{\tau}$ is divergence-free and also satisfies $\int_{F_\tau} \tr^{g_{\tau}}(h) \dvol^{g_{\tau}} = 0$ because
	\begin{gather*}
		\tr^{g_{\tau}}(\Lie_W g_{\tau})  = 2\div^{g_{\tau}}(W)
	\end{gather*}
	by~\eqref{eq:Lie} and this integrates to zero on a closed manifold.
	We observe that a non-zero multiple of $h$ is the derivative at $s = \tau$ of a family of Ricci-flat metrics $(\tilde{g}_s)_{s \in [0,\ell]}$ with $\tilde{g}_{s|s=\tau} = \vol(F_\tau,g_{\tau})^{-\frac{2}{n-1}} g_{\tau}$.
	In fact, such a family may be obtained by rescaling $(g_s)_{s \in [0,\ell]}$ to unit volume and suitably deforming the result with the flow of $W$.
	It follows that $h$ satisfies the \emph{linearized Einstein equation} $E^{\prime}_{\tilde{g}_{\tau}}(h) = 0$ and defines an \emph{infinitesimal Einstein deformation} of $\tilde{g}_{\tau}$ in the sense of \cite[Def.~12.29]{Besse:1987}.
	From \cite[Thm.~12.30]{Besse:1987} we get that $h$ is a TT-tensor, \ie is divergence- and trace-free (pointwise), with respect to $\tilde{g}_{\tau}$ and then also with respect to $g_{\tau}$.
	In particular, the j-equation $\div^{F_{\tau}}(h) - \upd \tr^{F_{\tau}} (h) = 0$ holds for $h$.
	Since it obviously also holds for $cg_{\tau}$, we are left to prove it for the longitudinal part $\Lie_W g_{\tau}$ of $\dot{g}_{\tau}$.
	
	In order to do so, we decompose $W$ using the Hodge decomposition $W^\flat = \upd f + \alpha + \delta\beta$ with $f \in C^\infty(F_\tau)$ and $\alpha \in \ker(\upd) \cap \ker(\delta) \subseteq \Omega^1(F_\tau)$ and $\beta \in \Omega^2(F_\tau)$.
	Due to \cref{Lem:DivPartHodge}, the exact and the harmonic\footnote{Actually, since harmonic forms on Ricci-flat manifolds are already parallel, this piece does not even affect $\Lie_W g_{\tau}$.} part satisfy the j-equation.
	We show that the co-exact part vanishes.
	For this, we observe that condition~\eqref{eq:ClosedCond}, which reduces to $-\upd \delta \upd \delta \beta = 0$, implies that 
	\begin{gather*}
		0 = \delta \upd \delta \upd \delta \beta = (\delta \upd + \upd \delta)(\delta \upd + \upd \delta) \delta \beta = \triangle^2 \delta \beta,
	\end{gather*}
	where $\triangle$ is the Hodge Laplacian.
	Hence $\delta \beta$ is harmonic, so $\delta \beta = 0$.
\end{proof}

\section{Local uniqueness of DEC extensions} \label{Sec:RigidSpacetime}
This section is devoted to the proofs of \cref{Thm:RigidSpacetime} and its corollaries.
The main step is proving the following theorem.

\begin{Satz} \label{Thm:ExParVFSpacetime}
	Let $(\overline{M}, \overline{g})$ be a time-oriented Lorentzian manifold subject to DEC and assume that the induced initial data set on an (embedded) spacelike hypersurface $M$ carries a future-lightlike $\overline{\nabla}$-parallel vector field $V = ue_0 - U \in \Gamma(\overline{T}M)$, $U \in \Gamma(TM)$.
	Suppose that the triple $(g,k,V)$ satisfies the following rigidity property:
	\begin{quote}
		\emph{For every $p \in M$ there is an open neighborhood $W \subsetneq M$ of $p$ such that every DEC initial data set $(g^\prime, k^\prime)$ coinciding with $(g,k)$ on $M \setminus W$ carries a future-lightlike $\overline{\nabla}$-parallel vector field $V^\prime$ coinciding with $V$ on $M \setminus W$.}
	\end{quote}
	Then there is an open neighborhood of $M$ in $\overline{M}$ onto which $V$ extends as parallel vector field.
\end{Satz}
\begin{Bem}
	The rigidity assumption in \cref{Thm:ExParVFSpacetime} is everything we need and is seemingly weaker than the one of \cref{Thm:RigidSpacetime}:
	Simply take some point $q \in M \setminus \{p\}$ and set $\tilde{W} = M \setminus \{q\}$ in the definition of “locally rigid”.
	In fact, however, they are equivalent.
	For the converse direction, choose some $W_0$ that fulfills the assumptions imposed on $W$ in \cref{Thm:ExParVFSpacetime}.
	Now given some open neighborhood $\tilde{W}$ of $p$ and choose $W \subseteq W_0 \cap \tilde{W}$ so small that $\Int(M \setminus W)$ is connected and with the property that $M \setminus W = \Clos(\Int(M \setminus W))$ (which is both true for small balls since we assume $\dim(M) = n \geq 2$).
	It is almost immediate to see that this $W$ does the job.
	The only crucial point to note is that while a priori the $\overline{\nabla}$-parallel vector fields $V^\prime$ and $V$ only have to coincide on $M \setminus W_0 \neq \emptyset$, they do so on all of $M \setminus W$ due to continuity and connectedness.
\end{Bem}
\begin{proof}[Proof of \cref{Thm:RigidSpacetime}]
	Let $\overline{\nabla}$ be the Levi-Civita connection of $(\overline{M}, \overline{g})$.
	From \cref{Thm:ExParVFSpacetime}, we get an extension of $V$ on a neighborhood of $M$ in $\overline{M}$ that is future-lightlike and parallel with respect to $\overline{\nabla}$.
	We denote the extension by $V$ as well and consider the flow of $V$.
	We may assume that the neighborhood of $M$ is tubular, so that the complement of $M$ decomposes into future and past components.
	Since $V$ lightlike and thus transversal to the spacelike hypersurface $M$, it defines a local diffeomorphism $\Phi$ from an open neighborhood of $\{0\} \times M$ in $\R \times M$ to an open neighborhood of $M$ in $\overline{M}$.
	We assume that $\Phi$ is chosen such that its domain is of the form $\bigcup_{p \in M} (I_p \times \{p\})$, where $I_p \subseteq \R$ is an open interval containing $0$ for each $p \in M$. 
	Restricting the codomain we may assume surjectivity of $\Phi$.
	Since flow lines of $V$ must cross $M$ from past to future, they can hit $M$ at most once and thus $\Phi$ is injective.
	So $\Phi$ is a diffeomorphism with $\upd \Phi(\frac{\del}{\del v}) = V$, where $v$ denotes the $\R$-coordinate of $\R \times M$.
	Along $\{0\} \times M$, the pullback metric is given by
	\begin{gather*}
	\Phi^*\overline{g} = \pi^{TM}(V)^\flat \otimes \upd v + \upd v \otimes \pi^{TM}(V)^\flat + g,
	\end{gather*}
	where $\pi^{TM}$ denotes the orthogonal projection on $TM$.
	Since $V$ is parallel and thus in particular a Killing vector field, the formula for $\Phi^*\overline{g}$ holds on the whole domain of $\Phi$.
\end{proof}
\begin{proof}[Proof of \cref{Thm:ExParVFSpacetime}]
	Let us first of all recall the canonical isomorphism $\overline{T}M \cong T\overline{M}_{|M}$ of vector bundles with metric and connection from the introduction.
	It follows that $V$ defines section of $T\overline{M} \to \overline{M}$ along $M \subseteq \overline{M}$ that is future-lightlike and parallel.
	We wish to construct a parallel extension to an open neighborhood of $M$ that we also denoted by $V$.
	The idea is now to slightly perturb the hypersurface $M$ in $\overline{M}$ and use the rigidity property to also find a lightlike and $\overline{\nabla}$-parallel vector field $V$ along the perturbed hypersurface.
	It then requires some technical work to see that these vector fields along the hypersurfaces are just the restrictions of a single vector field $V$ defined on a neighborhood of $M$ in $\overline{M}$, \ie that different (spacelike) hypersurfaces passing through the same point lead to the same vector in that point.
	
	Consider a point $p \in M$.
	Let $T$ be a future-timelike vector field defined on an open neighborhood of $p$ in $\overline{M}$.
	There is an $\epsilon > 0$ and a small compact neighborhood $K$ of $p$ in $M$ such that the flow of $T$ is defined on $(-\epsilon, \epsilon) \times K$.
	We equip $(-\epsilon, \epsilon) \times K$ with the metric obtained by pulling back along the flow.
	Since $K$ is compact, after possibly making $\epsilon$ smaller, we may assume that there exists some $C > 0$ such that for all $(t,q) \in (-\epsilon,\epsilon) \times K$, $\alpha \in \R$ and $X \in T_qM$ with $|\alpha| < C|X|_g$ the vector $\alpha \frac{\del}{\del t} + X \in T_{(t,q)}((-\epsilon, \epsilon) \times K)$ is spacelike.
	Consequently, any smooth function $f \colon K  \to (\-\epsilon, \epsilon)$ with $|\upd f|_g < C$ defines a spacelike hypersurface $\Graph(f) = \set{(f(q),q)}{q \in K}$ of $(-\epsilon,\epsilon) \times K$.
	
	We may assume that $K$ was chosen small enough so that it is contained in a neighborhood $W \subsetneq M$ around $p$ with the properties mentioned in the assumption.
	We now identify $(-\epsilon, \epsilon) \times K$ with its diffeomorphic image in $\overline{M}$ and define the vector field $V$ on a neighborhood of $p$ in $(-\epsilon, \epsilon) \times K$ as follows.
	For any compactly supported function $f \in C^\infty_c(K)$ with $|f| < \epsilon$ and $|\upd f|_g < C$ we consider the spacelike hypersurface $\Graph(f)$ extended by $M \setminus K$.
	This is canonically diffeomorphic to $M$ and the induced initial data set obviously coincides outside of $K$.
	The rigidity property allows to extend $V_{|M \setminus K}$ to $\Graph(f)$ such that it is a future-lightlike $\overline{\nabla}$-parallel hypersurface vector field.
	Since the set of \emph{admissible} functions, \ie functions $f \in C^\infty_c(K)$ with $|f| < \epsilon$ and $|\upd f|_g < C$, is star-shaped with respect to the zero function, this procedure allows to define $V$ on a neighborhood of $p$ -- once we have seen that the obtained vector $V$ at a point $(t,q)$ does not depend on the function $f$ with $f(q) = t$ that is used.
	
	To show this independence of $f$, we observe the following.
	Since the vector field $V$ along $\Graph(f)$ is $\overline{\nabla}$-parallel, its value at $(f(q), q)$ may be obtained via parallel transport along a curve in $\Graph(f)$ starting outside of $\supp(f)$ and ending at $(q,f(q))$.
	In particular, the vector $V_{(f(q),q)}$ is the same for two functions $f$ coinciding on a curve from $M \setminus K$ to $q$.
	Thus it would suffice to show that for any admissible $f_1$ and $f_2$ with $f_1(q) = f_2(q)$ there is an admissible function $f$ that coincides with each of them on such a curve.
	
	Locally around $q$ such a function can be constructed because in the model situation around $0$ in $\R^n$ a solution is given by $f_0(x_1, \ldots, x_n) = f_2(x_1, \ldots, x_n) + f_1(x_1, 0, \ldots, 0) - f_2(x_1,0, \ldots, 0)$.
	Note that $f_0 = f_1$ on $\R \times \{0\}$ and $f_0 = f_2$ on $\{0\} \times \R^{n-1}$.
	If the coordinates are chosen orthogonally in $q$ and such that $\frac{\del}{\del x_i} \in \ker \upd_q f_2$ for all $i = 2, \ldots, n$, then
	\begin{gather} \label{eq:EstFirstDer}
	|\upd_q f_0|_g \leq |\upd_q f_1|_g \leq \max \{|\upd_q f_1|_g, |\upd_q f_2|_g\}.
	\end{gather}
	Moreover, there is an estimate $\left|\frac{\del}{\del x_j}\frac{\del}{\del x_k} f_0 \right| \leq 3\max \{\|f_1\|_{C^2,x}, \|f_2\|_{C^2,x}\}$ for all $j,k = 1, \ldots n$, where the $C^2$-norm is the one induced by the coordinate system $x=(x_1,\ldots,x_n)$ around $q$.
	Let $C_0$ satisfy $\max_{q^\prime \in K, i =1,2} |\upd_{q^\prime} f_i|_g < C_0 < C$.
	Together with \eqref{eq:EstFirstDer} it follows that the function $f_0$ satisfies $|f_0| < \epsilon$ and $|\upd f_0|_g \leq C_0$ in a small open neighborhood $U_0$ around $q$, whose size may be determined from $f_1(q)$, $C_0$, $\max \{|\upd_q f_1|_g, |\upd_q f_2|_g\}$ and $3\max \{\|f_1\|_{C^2,x}, \|f_2\|_{C^2,x}\}$ alone.
	
	Now, we choose two paths $\gamma_1$ and $\gamma_2$ in $M$ connecting $q$ with $M \setminus K$.
	We suppose that they only intersect in $q$ and that in the chosen coordinates around $q$ they lie within $\R \times \{0\}$ and $\{0\} \times \R^{n-1}$, respectively.
	Let $(\chi_0, \chi_1, \chi_2)$ be a partition of unity subordinate to the open cover $(U_0, K \setminus \im(\gamma_2), K \setminus \im(\gamma_1))$ of $K$.
	By construction, $f \coloneqq \sum_{i=0}^2 \chi_i f_i \in C^\infty_c(K)$ satisfies $|f| < \epsilon$ and coincides with $f_i$ along $\gamma_i$ for $i = 1,2$.
	Noting that $\upd \chi_2 = -\upd \chi_0 - \upd \chi_1$ and $|f_0-f_2| \leq \|f_1-f_2\|_\infty$, we can estimate
	\begin{align*}
	|\upd f|_g &\leq |\upd \chi_0 (f_0-f_2)|_g + |\upd \chi_1 (f_1-f_2)|_g + \sum_{i_0}^2 \chi_i |\upd f_i|_g \\
	&\leq (|\upd \chi_0|_g + |\upd \chi_1|_g)\|f_1-f_2\|_\infty + C_0.
	\end{align*}
	Hence there is some $\delta > 0$ such that if $\|f_1-f_2\|_\infty < \delta$, then the constructed function $f$ is admissible.
	So in the case where $f_1$ and $f_2$ are close enough, we are done.
	
	In the general case, let $N \in \N$ satisfy $\|f_1-f_2\|_\infty < \delta N$.
	Considering the convex combinations $f_{1,j} = \frac{N-j}{N}f_1 + \frac{j}{N} f_2$ and $f_{2,j} = \frac{N-j-1}{N}f_1 + \frac{j+1}{N} f_2$ for $j = 0, \ldots, N-1$ instead of $f_1$ and $f_2$, respectively, we run the same construction as above with precisely the same choice of coordinates $x$ around $q$, neighborhood $U_0$ and partition of unity $(\chi_i)_{i=0,1,2}$.
	Notice that the function $f_{0,j}$ constructed in the first step still satisfies $|f_{0,j}| < \epsilon$ and $|\upd f_{0,j}|_g \leq C_0$ on $U_0$.
	This is due to the fact that the listed constants determining the size of $U_0$ play the same role for $f_{0,j}$ as they did for $f_0$.
	For instance, \eqref{eq:EstFirstDer} can be replaced with
	\begin{align*}
	|\upd_q f_{0,j}|_g^2 &= \left|\upd_q f_{1,j} \left(\frac{\del}{\del x_1} \right) \right|^2 + \sum_{i=2}^n \left|\upd_q f_{2,j} \left(\frac{\del}{\del x_i} \right) \right|^2 \\
	&\leq \left|\upd_q f_{1,j} \left(\frac{\del}{\del x_1} \right) \right|^2 + \sum_{i=2}^n \left|\upd_q f_{1,j} \left(\frac{\del}{\del x_i} \right) \right|^2 \\
	&= |\upd_q f_{1,j}|_g^2 \leq (\max \{|\upd_q f_1|_g, |\upd_q f_2|_g\})^2,
	\end{align*}
	where we used in the second step $\upd_q f_{2,j} \left(\frac{\del}{\del x_i} \right) = \frac{N-j-1}{N} \upd_q f_1 \left(\frac{\del}{\del x_i} \right)$ and $\upd_q f_{1,j} \left(\frac{\del}{\del x_i} \right) = \frac{N-j}{N} \upd_q f_1 \left(\frac{\del}{\del x_i} \right)$ for all $i=2, \ldots, n$.
	Now since $\|f_{1,j}-f_{2,j}\|_\infty < \delta$, we obtain an admissible function $f_j$ coinciding with $f_{1,j}$ along $\gamma_1$ and with $f_{2,j}$ along $\gamma_2$.
	Thus for all $j = 0, \ldots, N-1$ at $(f_1(q),q)$ the vector field $V$ is the same for $f_{1,j}$ and $f_{2,j}= f_{1,j+1}$.
	It follows that it is the same for the functions $f_1$ and $f_2$ we started with.
	
	Thus the above procedure gives a well-defined future-lightlike vector field $V$ locally around $p$, which  on $M$ coincides with the previously defined vector field.
	It is smooth since a smooth variation of the spacelike hypersurfaces leads to a smooth variation of the $\overline{\nabla}$-parallel vector field by ODE-theory.
	To see that $V$ is $\overline{\nabla}$-parallel, we just observe that in each point it is $\overline{\nabla}$-parallel in the directions tangent to a hypersurface $\Graph(f)$ (for an admissible $f$) passing through that point and that the tangent directions of such hypersurfaces collectively span its whole tangent space in $\overline{M}$.
	Finally, these local extensions of $V$ glue together to a lightlike and $\overline{\nabla}$-parallel vector field defined on a sufficiently small open neighborhood of $M$ in $\overline{M}$ since a $\overline{\nabla}$-parallel extension is necessarily unique.
\end{proof}

\begin{Bem}
	Actually, it is not important in the proof of \cref{Thm:ExParVFSpacetime} that the $\overline{\nabla}$-parallel vector field $V$ is future-lightlike.
	Also, the proof of \cref{Thm:RigidSpacetime} can be adapted to work for more general $V$ as long as $V$ is transversal to $M$.
	The main reason for restricting to the lightlike case was that all our examples of rigid initial data sets have lightlike $V$.
	Moreover, we do not expect much gain from formulating in greater generality. 
	At least in the causal case, where $V$ is guaranteed to be transversal to $M$, this can be seen as follows.
	First of all, the initial data set has to satisfy $\rho = |j|_g$ everywhere.
	Otherwise, there are many different local DEC spacetime extensions of the region where $\rho > |j|$, not all of them carrying a lightlike parallel vector field, contradicting the rigidity of $V$.
	Now assume that the initial data set extends to a DEC spacetime with $\overline{\nabla}$-parallel future-causal vector field $V$.
	By \cite[Lem.~3]{Gloeckle:2024_b}, its Einstein curvature has to be given by $\frac{1}{\rho}(\rho e_0^\flat - j) \otimes (\rho e_0^\flat - j)$ in all points of $M$ with $\rho \neq 0$.
	But now in all these points we have
	\begin{gather*}
		\frac{1}{\rho}\overline{g}(\rho e_0 - j^\sharp, V) \cdot (\rho e_0^\flat - j) = \ric^{\overline{g}}(V,\blank) - \frac{1}{2} \scal^{\overline{g}} \overline{g}(V,\blank) = -\frac{1}{2} \scal^{\overline{g}} \cdot V^\flat.
	\end{gather*}
	This can only be satisfied if $V$ is a multiple of the lightlike vector $\rho e_0 - j^\sharp$ (and $\scal^{\overline{g}} = 0$).
	Thus if $V$ is timelike, then we have to have $\rho \equiv 0$ and the spacetime has to be a vacuum spacetime.
	In this case -- as explained in the introduction -- local geometric uniqueness is already known as a consequence of the solution of the Cauchy problem for the vacuum Einstein equations and the conservation theorem of Hawking and Ellis.	
\end{Bem}

We are now in the position to prove the corollaries.
We start with \cref{Cor:LocGeomUniqSpin,Cor:LocGeomUniqMOTS}, which are very similar and deal with the situation of a band $[0,\ell] \times  F$ foliated by MOTS.
Afterwards, we consider closed manifolds -- mapping tori over $F$ -- in \cref{Cor:LocGeomUniqII}, which was not mentioned in the introduction.
Finally, we turn towards the (non-compact) asymptotically flat setting of \cref{Cor:LocGeomUniqAsympt}. 

\begin{proof}[Proof of \cref{Cor:LocGeomUniqSpin}]
	By \cref{Thm:RigidSpin}, $(M,g,k)$ carries a lightlike $\overline{\nabla}$-parallel spinor $\phi \in \Gamma(\overline{\Sigma}M)$.
	As explained in the introduction (\cf \cite{Gloeckle:2023p} for more details), it has an associated Dirac current $V_\phi \in \Gamma(\overline{T}M)$ that is future-lightlike and $\overline{\nabla}$-parallel.
	We fix a point $q \in \del_- M$ and consider the normalization $V = \frac{1}{u_{\phi}(q)} V_\phi$, where $V_\phi = u_\phi e_0 - U_\phi$.  
	We will then have $V = e_0 + \nu$ in the point $q$.
	
	We wish to apply \cref{Thm:RigidSpacetime} to $(g_{|M_0}, k_{|M_0}, V_{|M_0})$.
	In order to do so, we need to check the local rigidity property of the triple.
	Given $p \in M_0$ and a neighborhood $\tilde{W} \subseteq M$ of $p$, choose a neighborhood $W \subseteq \tilde{W}$ of $p$ so small that the closure of $W$ in $M$ is disjoint from $\del M$, the open subset $\Int(M \setminus W)$ is connected and $M \setminus W = \Clos(\Int(M \setminus W))$.
	Let $(g^\prime, k^\prime)$ be any DEC initial data set on $M_0$ coinciding with $(g_{|M_0}, k_{|M_0})$ on $M_0 \setminus W$.
	By the first condition on $W$, the initial data set $(g^\prime,k^\prime)$ extends uniquely by continuity to a DEC initial data set on $M$ that coincides with $(g,k)$ on $M \setminus W$.
	Now note that the assumptions of \cref{Thm:RigidSpin} are satisfied for this initial data set:
	It satisfies DEC and the future null expansion scalar $\theta^+$ along the boundary coincides with the one of $(g,k)$ because the initial data sets are the same on $\del M$.
	Since the conditions of \cref{Thm:RigidSpin} were assumed to hold for $(M,g,k)$, the null expansion scalar $\theta^+$ has the required sign and all the needed topological conditions on $M$ are also satisfied. 
	Thus the theorem is applicable and the same construction as above yields a future-lightlike $\overline{\nabla}$-parallel vector field $V^\prime$ on $M$ with $V^\prime = e_0 + \nu$ in $q$.
	It should be pointed out that the connection $\overline{\nabla}$ used here corresponds to $(g^\prime,k^\prime)$, but on $M \setminus W$ it coincides with the one of $(g,k)$, so there is no ambiguity in using the symbol $\overline{\nabla}$ there.
	Since $\Int(M \setminus W)$ is connected and $\overline{\nabla}$-parallel vector fields are uniquely determined along a smooth path by the vector at a single point on the path, $V$ and $V^\prime$ do not only coincide in $q$ but on all of $\Int(M \setminus W)$ and hence also on the closure $M \setminus W$ by continuity.
	Thus the rigidity property is satisfied.
	
	Now, \cref{Thm:RigidSpacetime} yields that both $\overline{M}_1$ and $\overline{M}_2$ contain open neighborhoods of $M_0$ that isometrically embed (with codimension $0$) into the same Killing development, and the embeddings are the same on $M_0$ and both preserve the time-orientation.
	We finish the proof by remarking that this Killing development indeed satisfies the DEC by \cref{Prop:KDSatDEC}.
\end{proof}

\begin{proof}[Proof of \cref{Cor:LocGeomUniqMOTS}]
	The combination of \cref{Thm:RigidMOTS} with \cref{Thm:ExParVF} yields a lightlike $\overline{\nabla}$-parallel vector field $V$ in the context of \cref{Thm:RigidMOTS} that can be normalized so that $V = e_0 + \nu$ in a previously fixed point $q \in \del_-M$.
	The rest of the proof can be taken verbatim from the one of \cref{Cor:LocGeomUniqSpin}.
\end{proof}

\begin{Kor} \label{Cor:LocGeomUniqII}
	Let $(g,k)$ be a DEC initial data set on a closed manifold $M$ and suppose that it contains a connected hypersurface $F$ that is a MOTS with respect to some unit normal $\nu$.
	Assume furthermore that either
	\begin{itemize}
		\item $M$ is spin and $\hat{A}(F) \neq 0$, or
		\item the manifold with boundary obtained by cutting $M$ along $F$ satisfies the homotopy condition with respect to the boundary piece (diffeomorphic to $F$) where $\nu$ is inward-pointing and in addition $F$ satisfies the cohomology condition.
	\end{itemize}
	Then there is a DEC spacetime extension of $(M,g,k)$ and this extension is locally geometrically unique.
\end{Kor}
\begin{Bem}
	Although the conditions of \cref{Cor:LocGeomUniqII} imply that after cutting $M$ along $F$, the initial data set carries a future-lightlike $\overline{\nabla}$-parallel vector field, this does not need be the case for the original initial data set $(M,g,k)$:
	There is no reason to believe that the vector field fulfills the fitting condition necessary for re-gluing.
	Therefore we cannot apply \cref{Thm:RigidSpacetime} directly to $M$.
\end{Bem}
\begin{proof}[Proof of \cref{Cor:LocGeomUniqII}]
	Cutting $M$ along $F$, we obtain a manifold whose boundary consists of two copies of $F$ -- one where $\nu$ is inward-pointing and another one where it points outwards.
	It comes with an initial data set obtained from $(g,k)$ by cutting, which satisfies the assumptions of Theorems~\labelcref{Thm:RigidSpin} or~\labelcref{Thm:RigidMOTS}.
	From these theorems, we obtain a diffeomorphism of the cut manifold to $[0,\ell] \times F$.
	The identification of the two boundaries in the original manifold $M$ provides us with a diffeomorphism $\phi \in \mathrm{Diff}(F)$ such that $M$ is diffeomorphic to the mapping torus of $\phi$.
	More precisely, the $M \cong (\R \times F)/\Z$, where the $\Z$-action is generated by $\sigma \colon (s,p) \mapsto (s + \ell, \phi^{-1}(p))$ with $s \in \R$, $p \in F$.
	
	Next we note that the induced $\Z$-invariant initial data set on the covering $\R \times F$ of $M$ carries a lightlike $\overline{\nabla}$-parallel vector field $V = ue_0 - U = u(e_0 +\nu)$, where as before $\nu$ is the unit normal on the canonical leaves pointing in positive $\R$-direction.
	The function $u$ is obtained from \cref{Thm:RigidSpin} or~\cref{Thm:RigidMOTS,Thm:ExParVF}, respectively, on the fundamental domain $[0,\ell] \times F$ of the covering $\R \times F \to M$.
	Since $F$ is connected, uniqueness of the parallel transport implies that there is a constant $\lambda \in \R_{>0}$ such that $u(\ell,p) = \lambda u(0,\phi(p))$ for all $p \in F$.
	Thus it extends to $\R \times F$ using the relation $u \circ \sigma = \lambda u$.
	The Killing development of the initial data set on $\R \times F$ with respect to $V$ satisfies the DEC by \cref{Prop:KDSatDEC}.
	Moreover, it has a compatible isometric $\Z$-action generated by $\overline{\sigma} \colon (v,s,p) \mapsto (\lambda^{-1} v, s + \ell, \phi^{-1}(p))$ for $v,s \in \R$, $p \in F$.
	Thus the quotient $(\R \times \R \times F)/\Z$ with the induced metric is a DEC spacetime extension of $(M,g,k)$.
	
	Now observe that the initial data set on $\R \times F$ together with the vector field $V$ from above satisfies the local rigidity property.
	For any point $(s,p) \in \R \times F$, this follows by applying \cref{Thm:RigidSpin} or~\cref{Thm:RigidMOTS,Thm:ExParVF}, respectively, to a suitable segment $[a,b] \times F$ containing $(s,p)$.
	We thus obtain local geometric uniqueness of the DEC spacetime extension of the initial data set on the covering manifold $\R \times F$.
	Now given two DEC spacetime extensions $(\overline{M}_1, \overline{g}_2)$ and  $(\overline{M}_2, \overline{g}_2)$ of $(M,g,k)$, we may restrict both to tubular neighborhoods of $M$ (which are then homotopy equivalent to $M$) and consider their coverings corresponding to $\R \times F \to M$.
	Due to local geometric uniqueness, these coverings contain small neighborhoods $W_1$ and $W_2$ of $\R \times F$, respectively, that are isometric to each other.
	Moreover, the isometry $\Phi \colon W_1 \to W_2$ can be chosen to fix $\R \times F$ and the time-orientation.
	These conditions imply that for all $k \in \Z$, the equation $\overline{\sigma}_2^k \circ \Phi = \Phi \circ \overline{\sigma}_1^k$ holds on all of $W_1 \cap \overline{\sigma}_1^{-k}(W_1)$, where $\overline{\sigma}_i$ denotes the canonical $\Z$-action on the covering of $\overline{M}_i$ for $i = 1,2$.
	Hence, the translates $(\overline{\sigma}_2^k \circ \Phi \circ \overline{\sigma}_1^{-k})_{k \in \Z}$ glue together and $\Phi$ descends to an isometry between open neighborhoods of $M$ in $(\overline{M}_1, \overline{g}_1)$ and $(\overline{M}_2, \overline{g}_2)$.
\end{proof}

\begin{proof}[Proof of~\cref{Cor:LocGeomUniqAsympt}]
	First of all, the existence of a DEC spacetime extension is already contained in \cref{Thm:RigidAsympt} as was remarked in the introduction.
	We may assume that $M$ has only one end.
	This part of \cref{Thm:RigidAsympt} is \cite[Thm.~3.1~(2)]{Hirsch.Zhang:2024p}.
	We already mentioned in the introduction that local geometric uniqueness holds for DEC spacetime extensions of vacuum initial data sets due to the conservation theorem \cite[Sec.~4.2]{Hawking.Ellis:1973} and local uniqueness in the solution of the vacuum Cauchy problem \cite{Foures-Bruhat:1952}.
	We may thus restrict our attention to the case $E = |P| > 0$ for otherwise $(M,g,k)$ is a vacuum initial data set (embedding into Minkowski space) by the classical positive mass theorem in the spin case \cite{Witten:1981}.

	We choose and fix an asymptotic chart $\Phi \colon M \setminus K \cong \R^n \setminus \Clos(B_R(0))$, where $K$ is a compact subset and $R>0$.
	Let $Ee_0 + P \in \R^{n,1}$ be the ADM total energy-momentum vector with respect to this chart.
	Furthermore, the chart allows us to canonically identify $\overline{T}_pM \cong \R^{n,1}$ for all $p \in M \setminus K$.
	We claim that on any connected open neighborhood of infinity, there can be at most one lightlike $\overline{\nabla}$-parallel vector field that is asymptotic to $Ee_0 + P$.
	Suppose there are two, $V_1$ and $V_2$.
	Then $V_1-V_2$ is also $\overline{\nabla}$-parallel and hence $p \mapsto \overline{g}_p(V_1(p)-V_2(p), V_1(p)-V_2(p))$ is (locally) constant.
	Since $V_1(p)-V_2(p) \lto 0$ for $p \lto \infty$ and $\overline{g}_p$ is asymptotic to the Minkowski metric on $\R^{n,1}$, this constant has to be zero.
	Hence $V_1 - V_2$ is lightlike or zero everywhere and this is only possible if $V_1$ and $V_2$ are scalar multiples of each other.
	But since they are both asymptotic to the non-zero vector $Ee_0 + P$, the scaling factor must be one, so they are equal.
	
	\Cref{Thm:RigidAsympt} states that $\frac{\del}{\del v}_{|M}$ is such a vector field.
	This part of the theorem is obtained from \cite{Hirsch.Zhang:2024p} as follows.
	First of all, by simultaneously rescaling the coordinates $v$ and $s$, it suffices to show that the corresponding coordinate vector field (or its negative) in~\cite[(5.24)]{Hirsch.Zhang:2024p} is asymptotic to a non-zero scalar multiple of $Ee_0 + P$.
	It equals the vector field $-\frac{\del}{\del \tau}$ in the Killing development~\cite[(5.23)]{Hirsch.Zhang:2024p}.
	We can now read off that tangential part of $\frac{\del}{\del \tau}$, what we called $-U$, is given by\footnote{In the notation of \cite{Hirsch.Zhang:2024p}; we would call it $\grad^g(s)$ since $u$ usually has a different meaning in the present paper.} $\grad^g(u)$.
	Hence we have to study $|\upd u|_g e_0 + \grad^g(u)$.
	Note that indeed this lightlike vector field is $\overline{\nabla}$-parallel:
	Due to \cite[Thm.~3.1~(1)]{Hirsch.Zhang:2024p} it satisfies the first equation of~\eqref{eq:OlPar}, which actually implies the second one in the lightlike case.
	Its asymptotic value equals the one of $|\upd u_\infty|_g e_0 + \grad^g(u_\infty)$ as $u - u_\infty \in C^{3,\alpha}_{1-q}$ with $u_\infty(p) \coloneqq \langle \frac{P}{E}, \Phi(p) \rangle$, $p \in M \setminus K$.
	Since $g$ is $C^{2,\alpha}_{-q}$-close to the pullback of the Euclidean metric, the asymptotic value is $e_0+ \frac{P}{E}$, as desired.
	
	Now we have everything at hand to prove the local rigidity property of $(g,k,\frac{\del}{\del v}_{|M})$ from \cref{Thm:RigidSpacetime}.
	We simply choose for any neighborhood $\tilde{W}$ of $p \in M$ a neighborhood $W \subseteq \tilde{W}$ with the properties that $W$ is bounded, the open subset $\Int(M \setminus W)$ is connected and $M \setminus W = \Clos(\Int(M \setminus W))$, which is the case for a sufficiently small ball around $p$.
	Then any DEC initial data set $(g^\prime,k^\prime)$ that coincides with $(g,k)$ outside of $W$ satisfies the same asymptotic conditions and so, by \cref{Thm:RigidAsympt}, also admits a lightlike $\overline{\nabla}$-parallel vector field $V^\prime$ asymptotic to $Ee_0 + P$.
	Since $\Int(M \setminus W)$ is a connected open neighborhood of infinity, $V^\prime$ has to coincide there with the original $\frac{\del}{\del v}_{|M}$ by the uniqueness property shown above.
	By continuity, it also coincides on all of $M \setminus W$.
	We may thus apply \cref{Thm:RigidSpacetime}, which completes the proof.
\end{proof}

\appendix	
\section{The spacetime DEC for pp-wave spacetimes}
In this appendix, we consider pp-wave spacetimes arising as the Killing development of a DEC initial data set with lightlike $\overline{\nabla}$-parallel vector field.
We show a sufficient criterion for them to satisfy the spacetime DEC.

\begin{Prop} \label{Prop:KDSatDEC}
	Let $(g,k)$ be a DEC initial data set on a manifold $M$.
	Assume that it carries a future-lightlike $\overline{\nabla}$-parallel vector field $V = ue_0 - U$ such that the induced metric on all the leaves of the foliation defined by $U^\perp$ is Ricci-flat.
	Then the Killing development~\eqref{eq:KD} of $(g,k,V)$ satisfies the (spacetime) DEC.
\end{Prop}
\begin{proof}
	Since the Killing development $(\overline{M} = \R \times M,\overline{g})$ carries a lightlike Killing vector field $\frac{\del}{\del v}$ (extending $V$), it suffices to check the condition along some Cauchy hypersurface.
	We choose $M = \{0\} \times M$ for this purpose.
	Let us first calculate the scalar curvature of $(\overline{M},\overline{g})$.
	We denote by $e_1, \ldots, e_{n-1}, \nu$ a local orthonormal frame of $TM$, where $\nu = -\frac{U}{u}$.
	Let $(F_\tau,g_\tau)$ be a leaf with its induced metric.
	For $p \in F_\tau$ and $X,Y,Z,W \in T_pF_\tau$ the Gauß equation implies
	\begin{align*}
	\overline{g}(\overline{R}(X,Y)Z,W) &= g_\tau(R^{g_\tau}(X,Y)Z,W) + \frac{1}{2}\overline{g}(\overline{\nabla}_X (e_0+\nu),W)\overline{g}(\overline{\nabla}_Y (e_0-\nu),Z)  \\
	&\phantom{=g_\tau(R^{g_\tau}(X,Y)Z,W)}\;+ \frac{1}{2}\overline{g}(\overline{\nabla}_X (e_0-\nu),W)\overline{g}(\overline{\nabla}_Y (e_0+\nu),Z) \\ &\phantom{=g_\tau(R^{g_\tau}(X,Y)Z,W)}\;-\frac{1}{2}\overline{g}(\overline{\nabla}_X (e_0+\nu),Z)\overline{g}(\overline{\nabla}_Y (e_0-\nu),W) \\
	&\phantom{=g_\tau(R^{g_\tau}(X,Y)Z,W)}\;-\frac{1}{2}\overline{g}(\overline{\nabla}_X (e_0-\nu),Z)\overline{g}(\overline{\nabla}_Y (e_0+\nu),W) \\
	&= g_\tau(R^{g_\tau}(X,Y)Z,W),
	\end{align*}
	where we used that $\overline{\nabla}(e_0+\nu) = -(\upd \log u) \cdot (e_0+\nu)$. 
	For any $X,Y \in T_pF_\tau$, we thus get
	\begin{equation} \label{eq:RicciVanish}
	\begin{aligned}
	\ric^{\overline{g}}(X,Y) &= -\frac{1}{2}\overline{g}(\overline{R}(X,e_0-\nu) (e_0+\nu), Y) - \frac{1}{2}\overline{g}(\overline{R}(X,e_0+\nu) (e_0-\nu), Y) \\
	&\phantom{=}\;+ \sum_{i=1}^{n-1} \overline{g}(\overline{R}(X,e_i) e_i, Y) \\
	&= \ric^{g_\tau}(X,Y) = 0
	\end{aligned}
	\end{equation}
	
	because $\overline{g}(\overline{R}(X,e_0+\nu) (e_0-\nu), Y) = \overline{g}(\overline{R}(Y,e_0-\nu) (e_0+\nu), X) = 0$ as $V = u(e_0+\nu)$ is $\overline{\nabla}$-parallel.
	Similarly, we obtain
	\begin{gather*}
	\scal^{\overline{g}} = -\ric^{\overline{g}}(e_0-\nu, e_0+\nu) + \sum_{i=1}^{n-1} \ric^{\overline{g}}(e_i, e_i) = 0.
	\end{gather*}
	Hence, Einstein and Ricci curvature of $(\overline{M},\overline{g})$ coincide.
	Now
	\begin{gather} \label{eq:MargDEC}
	0 = \ric^{\overline{g}}(e_0,e_0+\nu) = \rho + j(\nu) \geq \rho - |j| \geq 0
	\end{gather}
	shows that $j = -\rho \nu^\flat$.
	This implies $\ric^{\overline{g}}(e_0,X) = j(X) = 0$ for all $X \in TF_\tau$.
	Moreover, due to $\ric^{\overline{g}}(e_0+\nu,X) = 0$, we also have $\ric^{\overline{g}}(\nu,X) = 0$ for all $X \in TF_\tau$.
	So in view of~\eqref{eq:RicciVanish}, the only possibly non-zero components of the Einstein tensor with respect to $e_0,e_1, \ldots, e_{n-1},\nu$ are
	\begin{align*}
	\ric^{\overline{g}}(e_0,e_0) &= \rho, & \ric^{\overline{g}}(e_0,\nu) &=\ric^{\overline{g}}(\nu,e_0) = j(\nu) = -\rho &&\text{and} & \ric^{\overline{g}}(\nu,\nu) &= \rho,
	\end{align*}
	where the last equation results from $\ric^{\overline{g}}(\nu,e_0+\nu) = 0$.
	It is now apparent that the DEC is satisfied.
\end{proof}

\begin{Bem}
	The proof in particular shows that the Killing development of a DEC initial data set admitting a lightlike $\overline{\nabla}$-parallel vector field does not automatically satisfy the spacetime DEC.
	More precisely, let us assume that the induced metrics on the leaves are just scalar-flat.
	Then the proof essentially goes through in the same way; the only adaption being that there might be further non-vanishing components of the Einstein curvature given by  $\ric^{\overline{g}}(X,Y) = \ric^{g_\tau}(X,Y)$ for $X,Y \in TF_\tau$.
	Now, \cite[Lem.~3]{Gloeckle:2024_b} shows that DEC is satisfied if and only if $\ric^{g_\tau} \equiv 0$ for all $\tau$.
\end{Bem}

\section{\cref{Lem:Closed} from the point of view of the second Bianchi identity} \label{Sec:AppBianchi}
\Cref{Lem:Closed} may be seen as a consequence of the second Bianchi identity for an ambient Lorentzian manifold.
In this appendix, we give the corresponding proof that is a little bit longer than the one in the main text.
It starts with choosing an embedding of $(M,g,k)$ into a spacetime $(\overline{M}, \overline{g})$ such that $g$ is the induced metric on $M$ and $k$ is the induced second fundamental form on $M$ with respect to the future unit normal $e_0$.
This is always possible -- we are not imposing any energy condition on the Lorentzian manifold.
By its definition~\eqref{eq:DefLambda}, $\lambda$ is given in terms of the $(0,4)$-curvature tensor $\overline{R}$ of $(\overline{M},\overline{g})$ as $\lambda(X) = \overline{R}(X,\nu, e_0, \nu)$ for all $X \in TF_{\tau}$.
	
\begin{proof}[Alternative proof of~\cref{Lem:Closed}]
	In the following, let $X,Y$ be vector fields of $M$ that are tangent to the leaves of the foliation.
	We extend the vector fields $X, Y, e_0, \nu$ to $\overline{M}$ so that all the expressions in the following calculation make sense, although the equations themselves only hold on $F_{\tau}$ (where $\lambda$ is defined).
	First of all, we have
	\begin{align*}
	\upd \lambda (X,Y) &= \del_X \lambda(Y) - \del_Y \lambda(X) - \lambda([X,Y]) \\
	&= (\overline{\nabla}_X \overline{R}(\blank,\nu,e_0,\nu))(Y) - (\overline{\nabla}_Y \overline{R}(\blank,\nu,e_0,\nu))(X)		
	\end{align*}
	with
	\begin{align*}
	(\overline{\nabla}_X \overline{R}(\blank,\nu,e_0,\nu))(Y)
	&= (\overline{\nabla}_X \overline{R})(Y,\nu,e_0,\nu) + \overline{R}(Y,\overline{\nabla}_X \nu,e_0,\nu) \\
	&\phantom{=}\;+ \overline{R}(Y,\nu,\overline{\nabla}_X e_0,\nu)
	+ \overline{R}(Y,\nu,e_0,\overline{\nabla}_X \nu) \\
	&= (\overline{\nabla}_X \overline{R})(Y,\nu,e_0,\nu) + k(X,\nu) \overline{R}(Y,e_0,e_0,\nu),
	\end{align*}
	where the last equality is explained as follows.
	Since $\overline{R}(\blank,\blank,e_0,\nu) = \overline{R}(\blank,\blank,e_0+\nu,\nu)$ vanishes on $TF_\tau \otimes TF_\tau$ it is only the first summand of $\overline{\nabla}_X \nu = k(X,\nu) e_0 + \nabla_X \nu$ that contributes for the second term.
	Since only the $F_t$-tangential parts of $\overline{\nabla}_X e_0 = k(X,\blank)^\sharp$ and $\overline{\nabla}_X \nu$ contribute in the third and fourth term, respectively, the last two terms assemble to $\overline{R}(Y,\nu,e_0 + \nu, \nabla_X \nu) = 0$.
	
	Now, using the second Bianchi identity, we get
	\begin{align*}
	\upd \lambda (X,Y) &= -(\overline{\nabla}_{\nu} \overline{R})(X,Y,e_0,\nu) + k(X,\nu) \overline{R}(Y,e_0,e_0,\nu) - k(Y,\nu) \overline{R}(X,e_0,e_0,\nu) \\
	&= -\del_{\nu} \overline{R}(X,Y,e_0,\nu) + \overline{R}(\overline{\nabla}_{\nu} X,Y,e_0,\nu) + \overline{R}(X,\overline{\nabla}_{\nu} Y,e_0,\nu) \\
	&\phantom{=}\;+ \overline{R}(X,Y,\overline{\nabla}_{\nu} e_0,\nu) + \overline{R}(X,Y,e_0,\overline{\nabla}_{\nu} \nu) \\
	&\phantom{=}\;+ k(X,\nu) \overline{R}(Y,e_0,e_0,\nu) - k(Y,\nu) \overline{R}(X,e_0,e_0,\nu). 
	\end{align*}
	Similarly to before, the fourth and the fifth term add up to zero.
	Since $\overline{R}(X,Y,e_0,\nu) = \overline{R}(X,Y,e_0 + \nu,\nu) = 0$ on all of $M$, the first term is zero as well.
	For the second (and likewise the third) term, we observe that only the $F_\tau$-normal part $k(X,\nu) e_0 + g(\nabla_{\nu} X, \nu) \nu$ of $\overline{\nabla}_{\nu} X$ contributes, giving
	\begin{gather*}
	\overline{R}(\overline{\nabla}_{\nu} X,Y,e_0,\nu) = k(X,\nu) \overline{R}(e_0,Y,e_0,\nu) + \upd \log \strphi(X) \overline{R}(\nu,Y,e_0,\nu)
	\end{gather*}
	in view of $g(\nabla_{\nu} X, \nu) = -g(X, \nabla_{\nu} \nu) = k(X,\nu) = \upd \log \strphi(X)$ (\cf\eqref{eq:NablaNuNu}).
	Hence,
	\begin{gather*}
	\upd \lambda (X,Y) = -\upd \log \strphi(X) \overline{R}(Y,\nu,e_0,\nu) + \upd \log \strphi(Y) \overline{R}(X,\nu,e_0,\nu) = (\lambda \wedge \upd \log \strphi)(X,Y),
	\end{gather*}
	which readily implies~\cref{Lem:Closed}.	
\end{proof}


\begin{thebibliography}{99}
	{}
	\bibitem{Ammann.Gloeckle.Kroencke:2025pp}
	Bernd Ammann, Jonathan Glöckle, and Klaus Kröncke. “Construction of initial
	data sets for Lorentzian manifolds with lightlike parallel spinors”. In preparation.
	2025.
	{}
	\bibitem{Ammann.Kroencke.Mueller:2021}
	Bernd Ammann, Klaus Kröncke, and Olaf Müller. “Construction of initial data
	sets for {Lorentzian} manifolds with lightlike parallel spinors”. In:
	\emph{Commun. Math. Phys.} 387.1 (2021), pp. 77–109. \textsc{issn}:
	0010-3616. \textsc{doi}: \href {https://doi.org/10.1007/s00220-021-04172-1}
	{\nolinkurl {10.1007/s00220-021-04172-1}}.
	{}
	\bibitem{Beig.Chrusciel:1996}
	Robert Beig and Piotr T. Chruściel. “Killing vectors in asymptotically flat
	space–times. I. Asymptotically translational Killing vectors and the rigid positive
	energy theorem”. In: \emph{J. Math. Phys.} 37.4 (1996), pp. 1939–1961.
	\textsc{doi}: \href {https://doi.org/10.1063/1.531497} {\nolinkurl
		{10.1063/1.531497}}.
	{}
	\bibitem{Besse:1987}
	Arthur L. Besse. \emph{Einstein Manifolds}. Ergebnisse der Mathematik und
	ihrer Grenzgebiete, 3. Folge 10. Springer, 1987. \textsc{isbn}: 9783540741206.
	\textsc{doi}: \href {https://doi.org/10.1007/978-3-540-74311-8} {\nolinkurl
		{10.1007/978-3-540-74311-8}}.
	{}
	\bibitem{Cecchini.Zeidler:2024}
	Simone Cecchini and Rudolf Zeidler. “Scalar and mean curvature comparison via
	the Dirac operator”. In: \emph{Geom. Topol.} 28 (2024), pp. 1167–1212.
	\textsc{doi}: \href {https://doi.org/10.2140/gt.2024.28.1167} {\nolinkurl
		{10.2140/gt.2024.28.1167}}.
	{}
	\bibitem{Choquet-Bruhat:2009}
	Yvonne Choquet-Bruhat. \emph{General relativity and the Einstein equations}.
	Oxford Univerity Press, 2009. \textsc{isbn}: 9780199230723. \textsc{doi}: \href
	{https://doi.org/10.1093/acprof:oso/9780199230723.001.0001} {\nolinkurl
		{10.1093/acprof:oso/9780199230723.001.0001}}.
	{}
	\bibitem{Eichmair.Galloway.Mendes:2021}
	Michael Eichmair, Gregory J. Galloway, and Abraão Mendes. “Initial Data
	Rigidity Results”. In: \emph{Commun. Math. Phys.} 386 (2021), pp. 253–268.
	\textsc{doi}: \href {https://doi.org/10.1007/s00220-021-04033-x} {\nolinkurl
		{10.1007/s00220-021-04033-x}}.
	{}
	\bibitem{Foures-Bruhat:1952}
	Yvonne Fourès-Bruhat. “Théorème d’existence pour certains systèmes
	d’équations aux dérivées partielles non linéaires”. In: \emph{Acta Math.} 88
	(1952), pp. 141–225. \textsc{doi}: \href {https://doi.org/10.1007/BF02392131}
	{\nolinkurl {10.1007/BF02392131}}.
	{}
	\bibitem{Galloway:2008}
	Gregory J. Galloway. “Rigidity of marginally trapped surfaces and the topology
	of black holes”. In: \emph{Commun. Anal. Geom.} 16 (2008), pp. 217–229.
	\textsc{doi}: \href {https://doi.org/10.4310/CAG.2008.v16.n1.a7} {\nolinkurl
		{10.4310/CAG.2008.v16.n1.a7}}.
	{}
	\bibitem{Galloway.Mendes:2018}
	Gregory J. Galloway and Abraão Mendes. “Rigidity of marginally outer trapped
	2-spheres”. In: \emph{Comm. Anal. Geom.} 26 (2018), pp. 63–83. \textsc{doi}:
	\href {https://doi.org/10.4310/CAG.2018.v26.n1.a2} {\nolinkurl
		{10.4310/CAG.2018.v26.n1.a2}}.
	{}
	\bibitem{Galloway.Schoen:2006}
	Gregory J. Galloway and Richard Schoen. “A Generalization of Hawking’s Black
	Hole Topology Theorem to Higher Dimensions”. In: \emph{Commun. Math.
		Phys.} 266 (2006), pp. 571–576. \textsc{doi}: \href
	{https://doi.org/10.1007/s00220-006-0019-z} {\nolinkurl
		{10.1007/s00220-006-0019-z}}.
	{}
	\bibitem{Gloeckle:2023p}
	Jonathan Glöckle. “Initial data rigidity via {D}irac-{W}itten operators”.
	\textsc{arxiv}: \href {https://arxiv.org/abs/2304.02331} {\nolinkurl
		{2304.02331}}. 2023.
	{}
	\bibitem{Gloeckle:2024_b}
	Jonathan Glöckle. “Initial data sets with dominant energy condition admitting
	no smooth dec spacetime extension”. In: \emph{Proc. Amer. Math. Soc. Ser.~B}
	11 (2024), pp. 481–488. \textsc{doi}: \href {https://doi.org/10.1090/bproc/232}
	{\nolinkurl {10.1090/bproc/232}}.
	{}
	\bibitem{Gloeckle:2019}
	Jonathan Glöckle. “Initial Value Spaces in General Relativity”. Master Thesis.
	Universität Regensburg, Germany, 2019. \textsc{doi}: \href
	{https://doi.org/10.5283/epub.52853} {\nolinkurl {10.5283/epub.52853}}.
	{}
	\bibitem{Gloeckle:2024_c}
	Jonathan Glöckle. “Spinors and the Dominant Energy Condition for Initial Data
	Sets”. PhD Thesis. Universität Regensburg, Germany, 2024. \textsc{doi}: \href
	{https://doi.org/10.5283/epub.53013} {\nolinkurl {10.5283/epub.53013}}.
	{}
	\bibitem{Haferman:2023}
	Eduardo Hafemann. “Geometry and topology of black hole horizons”. Master
	Thesis. Universidade Federal de Santa Catarina, Brasil, 2023. \textsc{url}: \url
	{https://repositorio.ufsc.br/handle/123456789/251595}.
	{}
	\bibitem{Hawking.Ellis:1973}
	Stephen W. Hawking and George F. R. Ellis. \emph{The large scale structure of
		space-time}. Cambridge Monographs on Mathematical Physics, No. 1. Cambridge
	University Press, London-New York, 1973. \textsc{isbn}: 9780521099066.
	\textsc{doi}: \href {https://doi.org/10.1017/CBO9780511524646} {\nolinkurl
		{10.1017/CBO9780511524646}}.
	{}
	\bibitem{Hirsch.Zhang:2024p}
	Sven Hirsch and Yiyue Zhang. “Initial data sets with vanishing mass are
	contained in pp-wave spacetimes”. \textsc{arxiv}: \href
	{https://arxiv.org/abs/2403.15984} {\nolinkurl {2403.15984}}. 2024.
	{}
	\bibitem{Lee:2019}
	Dan A. Lee. \emph{Geometric Relativity}. Vol. 201. Graduate Studies in
	Mathematics. American Mathematical Society, 2019. \textsc{isbn}:
	9781470450816. \textsc{doi}: \href
	{https://doi.org/10.1365/s13291-022-00245-9} {\nolinkurl
		{10.1365/s13291-022-00245-9}}.
	{}
	\bibitem{Mueller.Nowaczyk:2017}
	Olaf Müller and Nikolai Nowaczyk. “A universal spinor bundle and the
	{E}instein-{D}irac-{M}axwell equation as a variational theory”. In: \emph{Lett.
		Math. Phys.} 107 (2017), pp. 933–961. \textsc{doi}: \href
	{https://doi.org/10.1007/s11005-016-0929-4} {\nolinkurl
		{10.1007/s11005-016-0929-4}}.
	{}
	\bibitem{ONeill:1983}
	Barrett O’Neill. \emph{Semi-{R}iemannian geometry. With applications to
		relativity}. Vol. 103. Pure and Applied Mathematics. Academic Press, 1983.
	{}
	\bibitem{Raede:2023}
	Daniel Räde. “Scalar and mean curvature comparison via $\mu $-bubbles”. In:
	\emph{Calc. Var.} 62.187 (2023). \textsc{doi}: \href
	{https://doi.org/10.1007/s00526-023-02520-8} {\nolinkurl
		{10.1007/s00526-023-02520-8}}.
	{}
	\bibitem{Witten:1981}
	Edward Witten. “A new proof of the positive energy theorem”. In:
	\emph{Comm. Math. Phys.} 80.3 (1981), pp. 381–402. \textsc{doi}: \href
	{https://doi.org/10.1007/BF01208277} {\nolinkurl {10.1007/BF01208277}}.
\end{thebibliography}
\end{document}